\documentclass{daj}

\usepackage[latin1]{inputenc}
\usepackage[english]{babel}
\usepackage[T1]{fontenc}
\usepackage{amsfonts}
\usepackage{amsmath}
\usepackage{amssymb}
\usepackage{stmaryrd}
\usepackage{amsthm}
\usepackage{enumerate}
\usepackage{pgf,tikz}
\usetikzlibrary{decorations.pathreplacing, shapes.multipart, arrows, matrix, shapes}
\usepackage[margin=0pt]{caption}

\newtheorem{lemme}{Lemma}
\newtheorem{theoreme}[lemme]{Theorem}
\newtheorem{prop}[lemme]{Proposition}
\newtheorem{coro}[lemme]{Corollary}

\newtheorem{theo}{Theorem}

\theoremstyle{definition}
\newtheorem{definition}[lemme]{Definition}

\theoremstyle{remark}
\newtheorem{rem}[lemme]{Remark}

\newtheorem{ex}[lemme]{Example}

\newcommand{\N}{\mathbf{N}}

\newcommand{\R}{\mathbf{R}}

\newcommand{\T}{\mathbf{T}}

\newcommand{\Q}{\mathbf{Q}}
\newcommand{\Z}{\mathbf{Z}}
\newcommand{\varep}{\varepsilon}
\newcommand{\Hom}{\operatorname{Homeo}}

\newcommand{\Leb}{\operatorname{Leb}}
\newcommand{\Diff}{\operatorname{Diff}}
\newcommand{\ud}{\,\operatorname{d}}
\newcommand{\Prb}{\mathcal{P}}

\newcommand{\card}{\operatorname{Card}}
\newcommand{\dist}{\operatorname{dist}}

\newcommand{\1}{\mathbf 1}

\newcommand{\Id}{\operatorname{Id}}
\newcommand{\im}{\operatorname{im}}

\newcommand{\Sp}{\mathbf{S}}

\newcommand{\ind}{{\boldsymbol{i}}}

\newcommand{\len}{\operatorname{len}}
\newcommand{\fat}{\operatorname{fat}}
\newcommand{\Ll}{\mathcal{L}}
\newcommand{\E}{\mathbf{E}}

\newcommand{\Dr}{D_{\text{rec}}}

\dajAUTHORdetails{%
  title = {Degree of Recurrence of Generic Diffeomorphisms}, 
  author = {Pierre-Antoine Guih{\'e}neuf},
  plaintextauthor = {Pierre-Antoine Guiheneuf},
  plaintexttitle = {Degree of Recurrence of Generic Diffeomorphisms}, 
  keywords = {Discretizations, generic dynamics, quasicrystals},
}   

\dajEDITORdetails{%
   year={2019},
   number={1},
   received={30 April 2018},   
   published={28 February 2019},  
   doi={10.19086/da.7545},       
}   

\begin{document}

\begin{frontmatter}[classification=text]


\author[pag]{Pierre-Antoine Guih{\'e}neuf\thanks{Supported by a CAPES/IMPA grant.}}

\begin{abstract}
We study spatial discretizations of dynamical systems: is it possible to recover some dynamical features of a system from numerical simulations? Here, we tackle this issue for the simplest algorithm possible: we compute long segments of orbits with a fixed number of digits. We show that for every $r\ge 1$, the dynamics of the discretizations of a $C^r$ generic conservative diffeomorphism of the torus is very different from that observed in the $C^0$ regularity. The proof of our results involves in particular a local-global formula for discretizations, as well as a study of the corresponding linear case, which uses ideas from the theory of quasicrystals.
\end{abstract}
\end{frontmatter}


%
%

\section{Introduction}

This paper is concerned with the issue of numerical simulations of dynamical systems. Consider a discrete-time dynamical system $f$ on the torus $\T^n = \R^n/\Z^n$ endowed with Lebesgue measure\footnote{We will see in Appendix~\ref{AddendSett} a more general framework where our results remain true.}, and define some spatial discretizations of this system in the following way: let $(E_N)_{N\in\N}$ be the collection of uniform grids on $\T^n$
\[E_N = \left\{ \left(\frac{i_1}{N},\cdots,\frac{i_n}{N}\right)\in \R^n/\Z^n \middle\vert\ 1\le i_1,\cdots,i_n \le N\right\},\]
and take a Euclidean projection $P_N$ on the nearest point of $E_N$; in other words $P_N(x)$ is (one of) the point(s) of $E_N$ which is the closest from $x$. This projection allows one to define the discretizations of $f$.

\begin{definition}
The \emph{discretization} $f_N : E_N\to E_N$ of $f$ on the grid $E_N$ is the map $f_N = P_N\circ f_{|E_N}$.
\end{definition}

Such discretizations $f_N$ are supposed to reflect what happens when segments of orbits of the system $f$ are computed numerically: in the particular case $N=10^k$, the discretization models computations made with $k$ decimal places. The question we are interested in is then the following: can the dynamics of the system $f$ be inferred from the dynamics of (some of) its discretizations $f_N$?
\bigskip

Here, we will focus on a number associated with the combinatorics of the discretizations, called the degree of recurrence. As every discretization is a finite map $f_N : E_N\to E_N$, each of its orbits is eventually periodic. Thus, the sequence of sets $(f_N^k(E_N))_{k\in\N}$ (the order of discretization $N$ being fixed) is nested and eventually constant, equal to a set $\Omega(f_N)$ called the \emph{recurrent set}\footnote{Note that a priori, $(f_N)^k\neq (f^k)_N$.}. This set coincides with the union of the periodic orbits of $f_N$; it is the biggest subset of $E_N$ on which the restriction of $f_N$ is a bijection. Then, the \emph{degree of recurrence} of $f_N$, denoted by $\Dr(f_N)$, is the ratio between the cardinality of $\Omega(f_N)$ and the cardinality of the grid $E_N$.

\begin{definition}\label{DefDegree}
Let $E$ be a finite set and $\sigma : E\to E$ be a finite map on $E$. The \emph{recurrent set} of $\sigma$ is the union $\Omega(\sigma)$ of the periodic orbits of $\sigma$. The \emph{degree of recurrence} of the finite map $\sigma$ is the ratio\index{$D(f_N)$}
\[\Dr(\sigma) = \frac{\card\big(\Omega(\sigma)\big)}{\card(E)}.\]
\end{definition}

This degree of recurrence $\Dr(f_N)\in [0,1]$ represents the loss of information induced by the iteration of the discretization $f_N$ (for example, if $\Dr(f_N)=1$, then $f_N$ is a bijection). Also, a finite map with a degree of recurrence equal to 1 preserves the uniform measure on $E_N$, and thus can be considered as \emph{conservative}.
\bigskip

The goal of this paper is to study the behaviour of the degree of recurrence $\Dr(f_N)$ as $N$ goes to infinity and for a \emph{generic} dynamics $f$ of the torus $\T^n$, $n\ge 2$. More precisely, on every Baire space $\mathcal B$ it is possible to define a good notion of genericity: a property on elements of $\mathcal B$ will be said \emph{generic} if satisfied on at least a countable intersection of open and dense subsets of $\mathcal B$.

For our purpose, the spaces $\Diff^r(\T^n)$ and $\Diff^r(\T^n,\Leb)$ of respectively $C^r$-diffeomorphisms and Lebesgue measure preserving $C^r$-diffeomorphisms of $\T^n$ are Baire spaces for every $r\in[0,+\infty]$, when endowed with the classical metric on $C^r$-diffeomorphisms\footnote{For example, $d_{C^1}(f,g)=\sup_{x\in\T^n} d(f(x),g(x)) + \sup_{x\in\T^n} \|Df_x-Dg_x\|$.}. Elements of $\Diff^r(\T^n)$ will be called \emph{dissipative} and elements of $\Diff^r(\T^n,\Leb)$ \emph{conservative}. Also, the spaces $\Diff^0(\T^n)$ and $\Diff^0(\T^n,\Leb)$ of homeomorphisms will be denoted by respectively $\Hom(\T^n)$ and $\Hom(\T^n,\Leb)$.

We will also consider the case of generic expanding maps of the circle. Indeed, for any $r\in [1,+\infty]$, the space $\mathcal D^r(\Sp^1)$ of $C^r$ expanding maps\footnote{Recall that a $C^1$ map $f$ is \emph{expanding} if for every $x\in\Sp^1$, $|f'(x)|>1$.} of the circle $\Sp^1$, endowed with the classical $C^r$ topology, is a Baire space.
 
The study of generic dynamics is motivated by the phenomenon of resonance that appears on some very specific examples like that of the linear automorphism of the torus $(x,y)\mapsto (2x+y,x+y)$: as noticed by \'E.~Ghys in \cite{Ghys-vari}, the fact that this map is linear with integer coefficients forces the discretizations on the uniform grids to be bijections with a (very) small global order (see also \cite{MR1176587} or the introduction of \cite{Guih-These}). One can hope that this behaviour is exceptional; to avoid them, \'E.~Ghys proposes to study the case of generic maps.
\bigskip

To begin with, let us recall what happens for generic homeomorphisms. The following consequence of \cite{MR3027586} can be found in Section 5.2 of \cite{Guih-discr}.

\begin{theoreme}[Guih{\'e}neuf]
Let $f\in\Hom(\T^n)$ be a generic dissipative homeomorphism. Then
\[\Dr(f_N)\underset{N\to+\infty}{\longrightarrow} 0.\]
\end{theoreme}

This theorem expresses that there is a total loss of information when a discretization of a generic homeomorphism is iterated. This is not surprising, as a generic dissipative homeomorphism has an ``attractor dynamics'' (see \cite{MR3027586}). The following result is the Corollary 5.24 of \cite{Guih-These} (see also Section 4.3 of \cite{Guih-discr} and Proposition 2.2.2 of \cite{Mier-dyna}).

\begin{theoreme}[Miernowski, Guih{\'e}neuf]\label{ThmGM}
Let $f\in\Hom(\T^n,\Leb)$ be a generic conservative homeomorphism. Then the sequence $\big(\Dr(f_N)\big)_{N\in\N}$ accumulates on the whole segment $[0,1]$.
\end{theoreme}

Thus, the degree of recurrence of a generic \emph{conservative} homeomorphism accumulates on the biggest set on which it can a priori accumulate. Then, the behaviour of this combinatorial quantity depends a lot on the order $N$ of the discretization and not at all on the dynamics of the homeomorphism $f$. Moreover, as there exists a subsequence $(N_k)_{k\in \N}$ such that $\Dr(f_{N_k})\to_{k\to+\infty} 0$, there exist ``a lot'' of discretizations that do not reflect the conservative character of the homeomorphism.
\bigskip

In this paper, we study the asymptotic behaviour of the degree of recurrence in higher regularity. Our results can be summarized in the following theorem (see Corollary~\ref{CoroCoroArturJairo}, Theorem~\ref{limiteEgalZero} and Corollary~\ref{Tau0Dilat}).

\begin{theo}\label{TheoIntro}
Let $r\in[1,+\infty]$ and $f$ be either
\begin{itemize}
\item a generic dissipative $C^1$-diffeomorphism of $\T^n$;
\item a generic conservative $C^r$-diffeomorphism of $\T^n$;
\item a generic $C^r$ expanding map\footnote{This point was previously announced by P.P. Flockerman and O. E. Lanford in 2008 \cite{LanfExpo}, but remained unpublished (see also \cite{Flocker}).} of the circle $\Sp^1$;
\end{itemize}
Then
\[\Dr(f_N)\underset{N\to+\infty}{\longrightarrow} 0.\]
\end{theo}

The two last items of this theorem are in fact valid under much weaker assumptions than Baire genericity. They hold under a countable number of explicit positive codimension conditions over the sequence of the differentials of the map. More precisely, these items are valid as soon as Thom's transversality theorem holds (see Lemma \ref{PerturbCr}); in particular they remain true for \emph{prevalent} maps (see \cite{HUNT201043} for a definition of prevalence and results about the validity of Thom's theorem).

We must admit that the degree of recurrence gives only little information about the discretizations' dynamics. Our model does not take into account the fact that most of computations are done with floating point numbers. Actually, most of simulations of dynamical systems are performed using more sophisticated algorithms and allow to compute accurately stable manifolds, SRB measures etc. However, our work describes theoretically what happens in ``most of cases'' for the naive algorithm, which has been only little studied theoretically up to now.

Theorem~\ref{TheoIntro} becomes interesting when compared to the corresponding case in $C^0$ regularity: in the conservative setting, it indicates that the bad behaviours observed for generic conservative homeomorphisms should disappear in any higher regularity: the behaviour of the rate of injectivity is less irregular for generic conservative $C^r$ diffeomorphisms than for generic conservative homeomorphisms (compare Theorems~\ref{ThmGM} and \ref{TheoIntro}). This big difference suggests that the wild behaviours of the discretizations global dynamics observed in \cite{Guih-discr} for generic conservative homeomorphisms may not appear in higher regularities.

More precisely, to any discretization $f_N$ and any $x\in\T^n$ one can associate the measure $\mu_x^{f_N}$ which is the uniform measure on the $f_N$-periodic orbit in which the positive orbit of $P_N(x)$ eventually falls. Oxtoby-Ulam theorem states that a generic conservative homeomorphism is ergodic \cite{Oxto-meas}, and it has been conjectured by A. Katok that this result is still true in $C^1$ regularity. Thus, one can expect the measures $\mu_x^{f_N}$ tend to Lebesgue measure for ``typical'' $x\in\T^n$.

In \cite{Guih-discr} it is proved that for a generic conservative homeomorphism $f$ and any $x\in\T^n$, these measures  $\mu_x^{f_N}$ accumulate (as $N$ goes to infinity) on the whole set of $f$-invariant measures; in \cite{Gui15c}, it is proved that the same holds for generic conservative $C^1$-diffeomorphisms. Hence, in both cases, the ergodic behaviour of a single orbit highly depends on the discretization order and do not tend to the ``physical'' dynamics of the map. For generic conservative homeomorphisms, the same result holds for the measures $\mu^{f_N}$, which are the averages of the measures $\mu_x^{f_N}$ over $x\in E_N$. However, some simulations suggest that this should no longer be true for generic conservative $C^1$-diffeomorphisms: the measures $\mu^{f_N}$ seem to tend to Lebesgue measure (see \cite{Gui15c}). From this one  can conjecture the following: the global dynamics of discretizations of generic conservative homeomorphism depends a lot on the order of discretization, but the global dynamics of discretizations of generic conservative $C^1$-diffeomorphism tends to the physical dynamics of the diffeomorphism (the dynamics of Lebesgue almost every point).

A particular case of this conjecture is answered positively by Theorems~\ref{ThmGM} and \ref{TheoIntro}. This can be seen as an indication that the measures $\mu^{f_N}$ indeed tend to Lebesgue measure in the $C^r$ generic case, $r\ge 1$.

Despite this, Theorem~\ref{TheoIntro} shows that while iterating the discretizations of a generic conservative diffeomorphism, a great amount of information is lost. Although $f$ is conservative, its discretizations tend to behave like dissipative maps. This can be compared with the work of P.~Lax \cite{MR0272983} (see also lemma 15 of \cite{Oxto-meas}): for any conservative homeomorphism $f$, there is a bijective finite map arbitrarily close to $f$. Theorem~\ref{TheoIntro} states that for a generic conservative $C^r$ diffeomorphism, the discretizations never possess this property.

\bigskip

However, the main interest of Theorem~\ref{TheoIntro} certainly lies in the techniques used to prove it: we will link global and local behaviours of the discretizations, thus reduce the proof to that of a linear statement.

As it can be obtained as the decreasing limit of finite time quantities, the degree of recurrence is maybe the easiest combinatorial invariant to study: we will deduce its behaviour from that of the rate of injectivity.

\begin{definition}\label{DefTauxDiffeo}
Let $n\ge 1$, $f:\T^n\to \T^n$ an endomorphism of the torus and $t\in\N$. The \emph{rate of injectivity} in time $t$ and for the order $N$ is the quantity\index{$\tau_N^t$}
\[\tau^t(f_N) = \frac{\card\big((f_N)^t(E_N)\big)}{\card(E_N)}.\]
Then, the \emph{upper rate of injectivity} of $f$ in time $t$ is defined as\index{$\tau^t$}
\begin{equation}\label{EqTauT}
\tau^t(f) = \limsup_{N\to +\infty} \tau^t(f_N),
\end{equation}
and the \emph{asymptotic rate of injectivity} of $f$ is 
\[\tau^\infty(f) = \lim_{t\to +\infty} \tau^t(f)\]
(as the sequence $(\tau^t(f))_t$ is decreasing, the limit is well defined).
\end{definition}

The link between the degree of recurrence and the rates of injectivity is made by the trivial formula:
\[\Dr(f_N) = \lim_{t\to+\infty}\tau^t(f_N).\]
Furthermore, for a fixed $N$, the sequence $(\tau^t(f_N))_t$ is decreasing in $t$, so $\Dr(f_N) \le \tau^t(f_N)$ for every $t\in\N$. Taking the upper limit in $N$, we get
\[\limsup_{N\to +\infty} \Dr(f_N) \le \tau^t(f)\]
for every $t\in\N$, so considering the limit $t\to+\infty$, we get
\begin{equation}\label{intervLim}
\limsup_{N\to +\infty} \Dr(f_N) \le \tau^\infty(f).
\end{equation}
In particular, if we have an upper bound on $\tau^\infty(f)$, this will give a bound on $\limsup_{N\to+\infty} \Dr(f_N)$. This reduces the proof of Theorem~\ref{TheoIntro} to the study of the asymptotic rate of injectivity $\tau^\infty(f)$.
\bigskip

For generic dissipative $C^1$-diffeomorphisms, the fact that the sequence $\big(\Dr(f_N)\big)_N$ converges to 0 (Theorem~\ref{TheoIntro}) is an easy consequence of Equation~\eqref{intervLim} and of a theorem of A.~Avila and J.~Bochi (Theorem~\ref{ArturJairo}, extracted from \cite{MR2267725}).

For conservative diffeomorphisms and expanding maps, the study of the rates of injectivity will be the opportunity to understand the local behaviour of the discretizations: we will ``linearize'' the problem and reduce it to a statement about generic sequences of linear maps.

Let us first define the corresponding quantities for linear maps: as in general a linear map does not send $\Z^n$ into $\Z^n$, we will approach it by a discretization. For any $A\in GL_n(\R)$ and any $x\in\Z^n$, $\widehat A(x)$ is defined as the point of $\Z^n$ which is the closest from $A(x)$. Then, for any sequence $(A_k)_k$ of linear maps, the \emph{rate of injectivity} $\tau^k(A_1,\cdots,A_k)$ is defined as the density of the set $(\widehat A_k\circ\cdots\circ\widehat A_1)(\Z^n)$ (see Definition~\ref{DefTaux}):
\[\tau^k(A_1,\cdots,A_k) = \limsup_{R\to +\infty} \frac{\card \big((\widehat{A_k}\circ\cdots\circ\widehat{A_1})(\Z^n) \cap B(0,R)\big)}{\card \big(\Z^n \cap B(0,R)\big)}\in]0,1],\]
and the \emph{asymptotic rate of injectivity} $\tau^\infty((A_k)_k)$ is the limit of $\tau^k(A_1,\cdots,A_k)$ as $k$ tends to infinity. These definitions are made to mimic the corresponding definitions for diffeomorphisms. The following statement asserts that that the rates of injectivity of a generic map are obtained by averaging the corresponding quantities for the differentials of the diffeomorphism, and thus makes the link between local and global behaviours of the discretizations (Theorem~\ref{convBisMieux}).

\begin{theo}\label{LocGlobIntro}
Let $r\in [1,+\infty]$, and $f\in \Diff^r(\T^n)$ (or $f\in \Diff^r(\T^n,\Leb)$) be a generic diffeomorphism. Then $\tau^k(f)$ is well defined (that is, the limit superior in \eqref{EqTauT} is a limit) and satisfies:
\[\tau^k(f) = \int_{\T^n} \tau^k\left(Df_x, \cdots, D f_{f^{k-1}(x)}\right) \ud \Leb(x).\]
\end{theo}

The same kind of result holds for generic expanding maps (see Theorem~\ref{TauxExpand}).

The proof of this theorem is quite technical but its general idea is very simple: apply an order 1 Taylor expansion and use some continuity results of the map $\tau^k$. These continuity results are proved using the model set formalism which will be discussed soon.

As noticed by Lanford in \cite{Lanf-inf}, ``\emph{this problem \emph{[of discretization]} reminds me quite a lot the notoriously difficult one of deriving non equilibrium statistical mechanics from atomic physics}''. As suggested by Hilbert in the statement of his $6^{\text{th}}$ problem \cite{MR1557926}, a fruitful approach is to study what happens at an intermediate scale between atomic or microscopic one (which here corresponds to the scale of the grid) and the macroscopic one (which here is the scale where we see the dynamics of the map)\footnote{For examples of recent progresses in this related field, see for instance \cite{MR3455156} and references therein.}. This scale is called mesoscopic and for us it will be the scale at which the map $f$ is almost affine: we will study discretizations of sequences of linear maps.
\bigskip

For conservative diffeomorphisms, the core of the proof of Theorem~\ref{TheoIntro} is the study of the rate of injectivity of generic sequences of matrices with determinant 1, that we will conduct in Section~\ref{SecLin}. Indeed, applying Theorem~\ref{LocGlobIntro} which links the local and global behaviours of the diffeomorphism, together with a transversality result (Lemma \ref{PerturbCr}), we reduce the proof of Theorem~\ref{TheoIntro} to the main result of Section~\ref{SecLin} (see Theorem~\ref{ConjPrincip} for the statement we actually need to prove Theorem~\ref{TheoIntro}).

\begin{theo}\label{ConjPrincipIntro}
For a generic sequence of matrices $(A_k)_{k\ge 1}$ of $\ell^\infty(SL_n(\R))$, we have
\[\tau^\infty\big( (A_k)_{k\ge 1}\big) = 0.\]
The same conclusion holds for a sequence of random iid matrices (Corollary \ref{CoroiidMat}) and for the iterations of a single matrix (Corollary \ref{CoroMemMat}) in both groups $SL_n(\R)$ and $O_n(\R)$.
\end{theo}

This linear statement has many consequences. First, it has nice applications to image processing. For example --- the result remains true for a generic sequence of \emph{isometries} ---, it says that if we apply a naive algorithm, the quality of a numerical image will be necessarily deteriorated by rotating this image many times by a generic sequence of angles. This gives an alternative proof of the result of \cite{Gui15b}\footnote{Which for its part is based on a Minkowski theorem for almost periodic patterns, see \cite{MR3690663}.} that may lead to new research directions for the field which studies discretizations of isometries of $\R^n$. For example, it could be used to set new discrete rotation algorithms, or to generalize the study of the local structure of the image sets to multiple iterations. See for instance \cite{MR2162013,thibault:tel-00596947, nouvel:tel-00444088, A1996_1075, ThesePluta}. Theorem~\ref{ConjPrincipIntro} is also the central technical result of \cite{Gui15c}, where it is combined with perturbation results in $C^1$ topology\footnote{Connecting lemma, ergodic closing lemma and some ad hoc perturbation techniques.} to study the physical measures of discretizations of generic conservative $C^1$-diffeomorphisms. 

The proof of Theorem~\ref{ConjPrincipIntro} is the most difficult and original part of this article. It uses the nice formalism of model sets, developed initially for the study of quasicrystals. Unfortunately, the classical theory does not apply here: the common definition of a model set includes a hypothesis of injectivity of a projection, which is not satisfied in the present setting\footnote{For complements about this more general class of model sets, see \cite{Gui15d}.} (as this non injectivity is the very source of the discretizations' loss of injectivity). However, using an equidistribution property, the model set viewpoint allows to get a geometric formula for the computation of the rate of injectivity of a generic sequence of matrices (Proposition~\ref{CalculTauxModel}): the rate of injectivity of a generic sequence $A_1,\cdots,A_k$ of matrices of $SL_n(\R)$ can be expressed in terms of areas of intersections of cubes in $\R^{nk}$. This formula, by averaging what happens in the image sets $(\widehat A_k\circ\cdots\circ\widehat A_1)(\Z^n)$, reflects the global behaviour of these sets. Also, it transforms the iteration into a passage in high dimension. This allows to prove Theorem~\ref{ConjPrincipIntro}, using quite technical geometric considerations about the volume of intersections of cubes, without having to make ``clever'' perturbations of the sequence of matrices (that is, the perturbations made a each iteration are chosen independently from that made in the past or in the future).

Interestingly, Lanford notes in \cite {Lanf-inf} that ``\emph{The problem \emph{[of discretization]} can probably be made much easier by the judicious introduction of a stochastic element in the microscopic evolution.[...] I think this is cheating.}'' Indeed, the equidistribution property we will crucially use during the proofs is equivalent to the uniform distribution of roundoff errors; in that sense it can be said that discretizations behave like random perturbations. This idea will be developed in a forthcoming paper in collaboration with M.~Monge \cite{GuihMonge}, in which we will use crucially the ideas of the present article to study the ergodic global middle-term behaviour of the discretizations of generic expanding maps.
\bigskip

The end of this paper is devoted to the results of the simulations we have conducted about the degree of recurrence of $C^1$-diffeomorphisms and expanding maps; it shows that in practice, the degree of recurrence tends to 0, at least for the examples of diffeomorphisms we have tested.
\bigskip

Recall that we will see in Appendix~\ref{AddendSett} that the quite restrictive framework of the torus $\T^n$ equipped with the uniform grids can be generalized to arbitrary manifolds, provided that the discretizations grids behave locally (and almost everywhere) like the canonical grids on the torus.
\bigskip

To finish, we state some questions related to Theorem~\ref{TheoIntro} that remain open.
\begin{itemize}
\item What is the behaviour of the degree of recurrence of discretizations of generic $C^r$-expanding maps of the torus $\T^n$, for $n\ge 2$ and $r>1$? In the view of Theorem~\ref{TheoIntro}, we can conjecture that this degree of recurrence tends to 0. To prove it we would need a generalization of Lemma~\ref{Espoir} to higher dimensions.
\item What is the behaviour of the degree of recurrence of discretizations of generic dissipative $C^r$-diffeomorphisms of the torus $\T^n$ for $r>1$? This question could reveal quite hard, as it may require some perturbation results in the $C^r$ topology for $r>1$.
\end{itemize}

\section{Discretizations of sequences of linear maps}\label{SecLin}

We begin by the study of the linear case, corresponding to the ``local behaviour'' of $C^1$ maps. We first define the linear counterpart of discretizations.

\begin{definition}\label{DefDiscrLin}
The map $P : \R\to\Z$\index{$P$} is defined as a projection from $\R$ onto $\Z$. More precisely, for $x\in\R$, $P(x)$ is the unique\footnote{The choice of where the inequality is strict or large is arbitrary.} integer $k\in\Z$ such that $k-1/2 < x \le k + 1/2$. This projection induces the map\index{$\pi$}
\[\begin{array}{rrcl}
\pi : & \R^n & \longmapsto & \Z^n\\
 & (x_i)_{1\le i\le n} & \longmapsto & \big(P(x_i)\big)_{1\le i\le n}
\end{array}\]
which is an Euclidean projection on the lattice $\Z^n$. For $A\in M_n(\R)$, we denote by $\widehat A$ the \emph{discretization}\index{$\widehat A$} of $A$, defined by 
\[\begin{array}{rrcl}
\widehat A : & \Z^n & \longrightarrow & \Z^n\\
 & x & \longmapsto & \pi(Ax).
\end{array}\]
\end{definition}

This definition allows us to define the rate of injectivity for sequences of linear maps.

\begin{definition}\label{DefTaux}
For a discrete subset $E\subset \R^n$, we denote\footnote{By definition, $B_R = B_\infty(0,R)$, where the considered norm is $\|x\|_\infty = \max(|x_1|,\cdots,|x_n|)$. Our results do not depend on the choice of the norm, but this one simplifies the computations.}
\[D_R = \frac{\card \big(E \cap B_R\big)}{\card \big(\Z^n \cap B_R\big)}.\]

Let $A_1,\cdots,A_k \in GL_n(\R)$. The \emph{rate of injectivity} of $A_1,\cdots,A_k$ is the asymptotic density\footnote{In the sequel we will see that the $\limsup$ is in fact a limit for a generic sequence of matrices. It can also be shown that it is a limit for any sequence of matrices, using the formalism developed hereafter and the classification of closed subgroups of $\T^n$.}\index{$\tau^k$}
\[\tau^k(A_1,\cdots,A_k) = \limsup_{R\to +\infty} D_R\Big((\widehat{A_k}\circ\cdots\circ\widehat{A_1})(\Z^n)\Big),\]
and for an infinite sequence $(A_k)_{k\ge 1}$ of invertible matrices, as the previous quantity is decreasing in $k$, we can define the \emph{asymptotic rate of injectivity}\index{$\tau^\infty$}
\[\tau^\infty\big((A_k)_{k\ge 1}\big) = \lim_{k\to +\infty}\tau^k(A_1,\cdots,A_k)\in[0,1].\]
\end{definition}

\begin{figure}[t]
\begin{minipage}[c]{.33\linewidth}
	\includegraphics[width=\linewidth, trim = 1.5cm .5cm 1.5cm .5cm,clip]{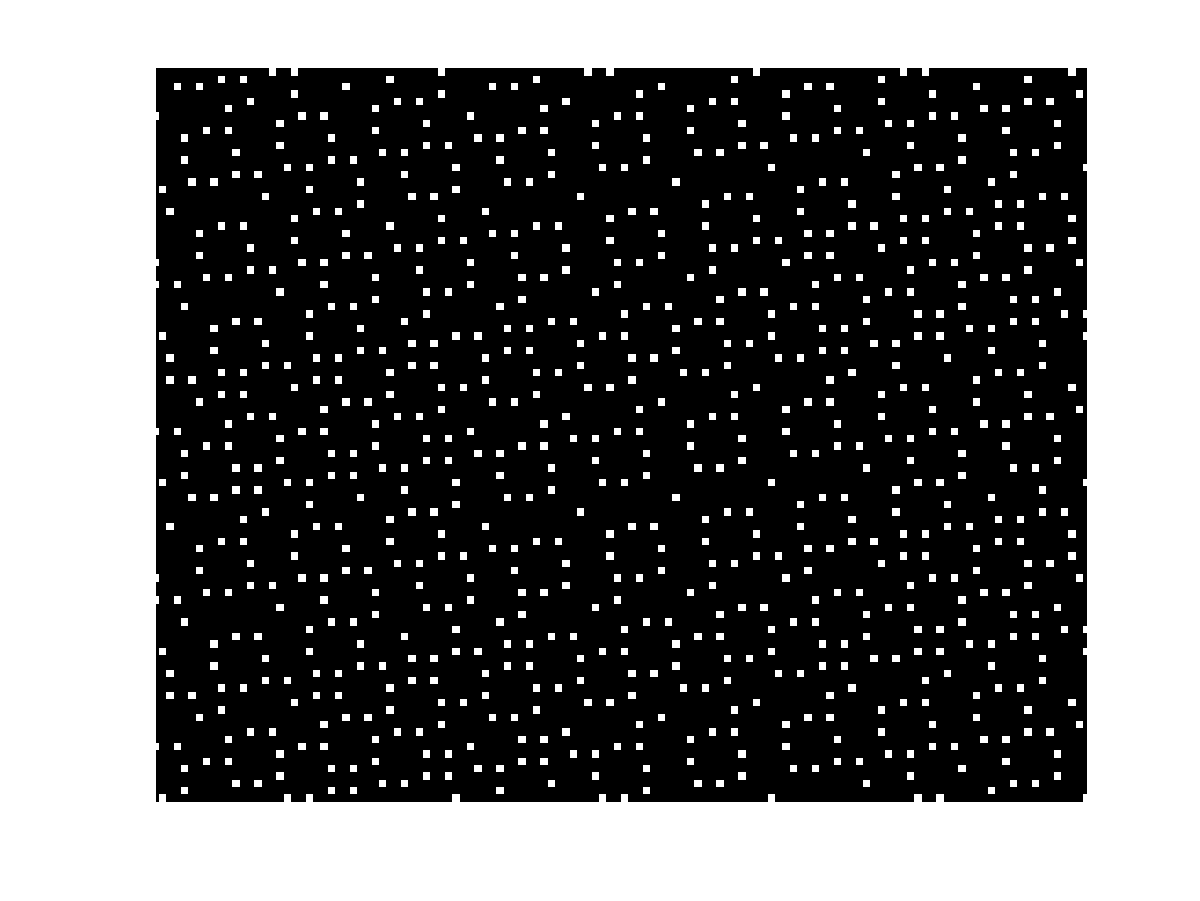}
\end{minipage}\hfill
\begin{minipage}[c]{.33\linewidth}
	\includegraphics[width=\linewidth, trim = 1.5cm .5cm 1.5cm .5cm,clip]{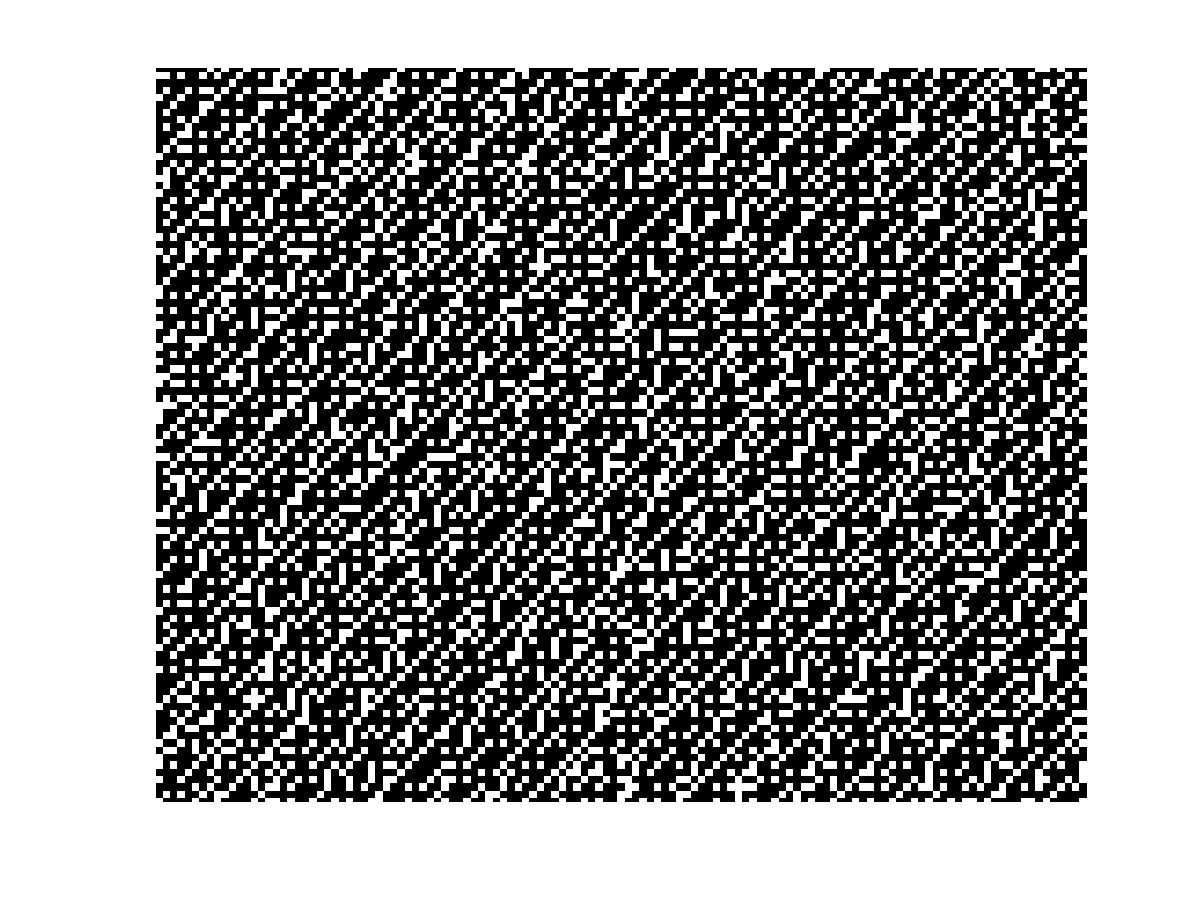}
\end{minipage}\hfill
\begin{minipage}[c]{.33\linewidth}
	\includegraphics[width=\linewidth, trim = 1.5cm .5cm 1.5cm .5cm,clip]{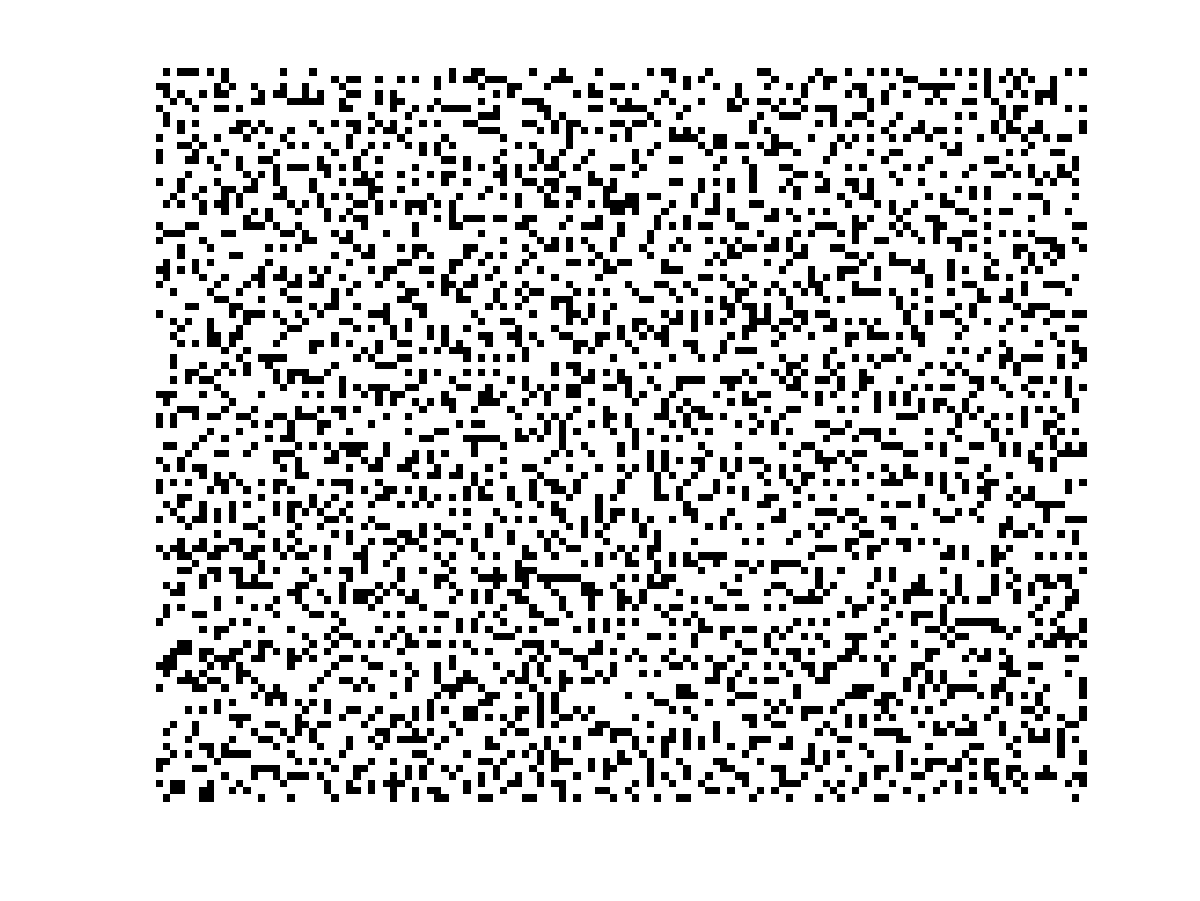}
\end{minipage}
\caption[Successive images of $\Z^2$ by discretizations of random matrices]{Successive images of $\Z^2$ by discretizations of random matrices in $SL_2(\R)$, a point is black if it belongs to $(\widehat{A_k}\circ\cdots\circ\widehat{A_1})(\Z^2)$. The $A_i$ are chosen randomly and independently, using the singular value decomposition: they are chosen among the matrices of the form $R_\theta D_t R_{\theta'}$, with $R_\theta$ the rotation of angle $\theta$ and $D_t$ the diagonal matrix $\operatorname{Diag}(e^t,e^{-t})$, the angles $\theta$, $\theta'$ being chosen uniformly in $[0,2\pi]$ and $t$ uniformly in $[-1/2,1/2]$. From left to right and top to bottom, $k=1,\, 3,\, 20$.}\label{ImagesSuitesMat}
\end{figure}

For a typical example of the sets $(\widehat{A_k}\circ\cdots\circ\widehat{A_1})(\Z^2)$, see Figure~\ref{ImagesSuitesMat}. Finally, we define a topology on the set of sequences of linear maps.

\begin{definition}\label{DefTopoSL}
We fix once for all a norm $\|\cdot\|$ on $M_n(\R)$. For a bounded sequence $(A_k)_{k\ge 1}$ of matrices of $SL_n(\R)$, we set\index{$\|(A_k)_k\|_\infty$}
\[\|(A_k)_k\|_\infty = \sup_{k\ge 1} \|A_k\|.\]
In other words, we consider the space $\ell^\infty(SL_n(\R))$ of uniformly bounded sequences of matrices of determinant 1 endowed with this classical metric.
\end{definition}

We take advantage of the rational independence between matrix coefficients of a generic sequence to obtain geometric formulas for the computation of the rate of injectivity. The tool used is the formalism of \emph{model sets}\footnote{Also called \emph{cut-and-project} sets.} (see for example \cite{Moody25} or \cite{MR2876415} for surveys about model sets, see also \cite{Gui15d} for the application to the specific case of discretizations of linear maps). Unfortunately, even if this viewpoint is very fruitful for the understanding of discretizations of linear maps, the classical theory of model sets do not apply here and we have to develop appropriate techniques for our specific setting.
\bigskip

Let us summarize the different notations we will use throughout this section. We will denote by $0^k$\index{$0^k$} the origin of the space $\R^k$, and $W^k = ]-1/2,1/2]^{nk}$ (unless otherwise stated). In this section, we will denote $B_R = B_\infty(0,R)$ and $D_c(E)$\index{$D_c$} the density of a ``continuous'' set $E\subset \R^n$, defined as (when the limit exists)
\[D_c(E) = \lim_{R\to+\infty} \frac{\Leb(B_R\cap E)}{\Leb(B_R)},\]
while for a discrete set $E\subset \R^n$, the notation $D_d(E)$\index{$D_d$} will indicate the discrete density of $E$, defined as (when the limit exists)
\[D_d(E) = \lim_{R\to+\infty} \frac{\card(B_R\cap E)}{\card(B_R\cap \Z^n)},\]

We will consider $(A_k)_{k\ge 1}$ a sequence of matrices of $SL_n(\R)$, and denote\index{$\Gamma_k$}
\[\Gamma_k = (\widehat{A_k}\circ\dots\circ\widehat{A_1}) (\Z^n).\]
Also, $\Lambda_k$ will be the lattice $M_{A_1,\cdots,A_k} \Z^{n(k+1)}$, with
\begin{equation}\label{DefMat1}
M_{A_1,\cdots,A_k} = \left(\begin{array}{ccccc}
A_1 & -I_n &        &        & \\
    & A_2  & -I_n   &        & \\
    &      & \ddots & \ddots & \\
    &      &        & A_k    & -I_n\\
    &      &        &        & I_n
\end{array}\right)\in M_{n(k+1)}(\R),
\end{equation}
and $\widetilde \Lambda_k$ will be the lattice $\widetilde M_{A_1,\cdots,A_k} \Z^{nk}$, with
\begin{equation}\label{DefMat2}
\widetilde M_{A_1,\cdots,A_k} = \left(\begin{array}{ccccc}
A_1 & -I_n &        &         & \\
    & A_2  & -I_n   &         & \\
    &      & \ddots & \ddots  & \\
    &      &        & A_{k-1} & -I_n\\
    &     &        &          & A_k
\end{array}\right)\in M_{nk}(\R).
\end{equation}
Finally, we will denote
\begin{equation}\label{DeftauBar}
\overline\tau^k(A_1,\cdots,A_k) = D_c\left( W^{k+1} + \Lambda_k \right)
\end{equation}
the \emph{mean rate of injectivity in time $k$} of $A_1,\cdots,A_k$, and 
\[\overline\tau^\infty\big((A_k)_k\big) = \lim_{k\to +\infty}\overline\tau^k(A_1,\cdots,A_k).\]

\subsection[A geometric viewpoint to compute the rate of injectivity]{A geometric viewpoint to compute the rate of injectivity in arbitrary times}

We begin by motivating the introduction of model sets by giving an alternative construction of the image sets $(\widehat{A_k}\circ\dots\circ\widehat{A_1}) (\Z^n)$ using this formalism.

Let $A_1,\cdots,A_k\in M_n(\R)$. A point $x\in\Z^n$ belongs to $\widehat{A_1}(\Z^n)$ if and only if there exists $y\in\Z^n$ such that $A_1y \in x+W^1$, that is $A_1y-x \in W^1$. In other words, denoting
\[\lambda = M_{A_1}\begin{pmatrix} y \\ x \end{pmatrix} = \begin{pmatrix}
A_1 y-x \\ x \end{pmatrix},\]
a point $x\in\Z^n$ belongs to $\widehat{A_1}(\Z^n)$ if and only if there exists $\lambda\in M_{A_1}\Z^{2n}$ such that $p_1(\lambda) = A_1 y-x \in W^1$ and $p_2(\lambda) = x$ (where $p_1$ is the projection on the $n$ first coordinates and $p_2$ the projection on the $n$ last). Thus, 
\[\Gamma_1 = \big\{p_2(\lambda)\mid \lambda\in \Lambda_1,\, p_1(\lambda)\in W^1\big\}\nonumber.\]

Iterating this idea, we get that
\begin{equation}\label{CalcGamma}
\Gamma_k = \big\{p_2(\lambda)\mid \lambda\in \Lambda_k,\, p_1(\lambda)\in W^k\big\} = p_2\Big(\Lambda_k \cap \big(p_1^{-1}(W^k)\big)\Big),
\end{equation}
with $p_1$ the projection on the $nk$ first coordinates and $p_2$ the projection on the $n$ last coordinates. This allows us to see the set $\Gamma_k$ as a model set.
\smallskip

\noindent\begin{minipage}[c]{.55\linewidth}
\hspace{1em} In general, a model set is a set of the form $p_2(\Lambda \cap (p_1^{-1}(W)))$, with $\Lambda$ a lattice of $\R^{m+n}$, $p_1 : \R^{m+n} \to \R^m$ and $p_2: \R^{m+n} \to \R^n$ the canonical orthogonal projections and $W$ some compact subset of $\R^m$. The classical definition also requires the injectivity of $p_{2|\Lambda}$ and the density of $p_1(\Lambda)$, the first of these two conditions being false in our context as soon as one of the maps $\widehat{A_i}$ is non-injective.
\end{minipage}\hfill
\begin{minipage}[r]{.4\linewidth}
\resizebox {\columnwidth} {!} {
\begin{tikzpicture}[scale=1]
\fill[color=blue!10!white] (-.6,-2) rectangle (.9,2);
\draw[color=blue!80!black] (-.6,-2) -- (-.6,2);
\draw[color=blue!80!black] (.9,-2) -- (.9,2);
\draw[color=blue!80!black, very thick] (-.6,0) -- (.9,0);
\draw[color=blue!80!black] (.25,0) node[below] {$W$};
\draw (-3.5,0) -- (3.5,0);
\draw (0,-2) -- (0,2);
\clip (-3.3,-2) rectangle (3.3,2);

\draw (.866,.364) -- (0,.364);
\draw (-.129,.987) -- (0,.987);
\draw (.737,1.351) -- (0,1.351);
\draw (-.258,1.974) -- (0,1.974);
\draw (.129,-.987) -- (0,-.987);
\draw (.258,-1.974) -- (0,-1.974);

\draw (0,0) node {$\times$};
\draw (0,.364) node {$\times$};
\draw (0,.987) node {$\times$};
\draw (0,1.351) node {$\times$};
\draw (0,1.974) node {$\times$};
\draw (0,-.987) node {$\times$};
\draw (0,-1.974) node {$\times$};

\foreach\i in {-4,...,4}{
\foreach\j in {-4,...,4}{
\draw[color=green!40!black] (.866*\i-.129*\j,.364*\i+.987*\j) node {$\bullet$};
}}
\draw[color=green!40!black] (1.2,-.8) node {$\Lambda$};
\end{tikzpicture}}
\end{minipage}
\smallskip

Under generic conditions among sequences of matrices, the set $p_1(\Lambda_k)$ is dense (thus, equidistributed) in the image set $\im p_1$. In particular, the set $\{p_1(\lambda)\mid \lambda\in\Lambda_k\}$ is equidistributed in the window $W^k$.

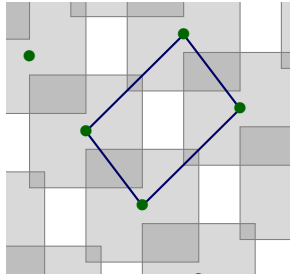
\begin{figure}[t]
\begin{center}
\begin{tikzpicture}[scale=1.5]
\clip (-1.2,-.6) rectangle (1.3,1.8);
\foreach\i in {-1,...,2}{
\foreach\j in {-1,...,3}{
\fill[color=gray,opacity = .3] (.866*\i-.5*\j-.5,.859*\i+.659*\j-.5) rectangle (.866*\i-.5*\j+.5,.859*\i+.659*\j+.5);
\draw[color=gray] (.866*\i-.5*\j-.5,.859*\i+.659*\j-.5) rectangle (.866*\i-.5*\j+.5,.859*\i+.659*\j+.5);
}}
\draw[color=blue!40!black,thick] (0,0) -- (.866,.859) -- (.366,1.518) -- (-.5,.659) -- cycle;
\foreach\i in {-2,...,2}{
\foreach\j in {-2,...,2}{
\draw[color=green!40!black] (.866*\i-.5*\j,.859*\i+.659*\j) node {$\bullet$};
}}
\end{tikzpicture}
\caption{Geometric construction to compute the rate of injectivity: the green points are the elements of $\Lambda$, the blue parallelogram is a fundamental domain of $\Lambda$ and the grey squares are centred on the points of $\Lambda$ and have radii $1/2$. The rate of injectivity is equal to the area of the intersection between the union of the grey squares and the blue parallelogram.}\label{tourp}
\end{center}
\end{figure}

The following property makes the link between the density of $\Gamma_k$ --- that is, the rate of injectivity of $A_1,\cdots,A_k$ --- and the density of the union of unit cubes centred on the points of the lattice $\Lambda_k$ (see Figure~\ref{tourp}). This formula seems to be very specific to the model sets defined by the matrix $M_{A_1,\cdots,A_k}$ and the window $W^k$, it is unlikely that it could be generalized to other model sets.

\begin{prop}\label{CalculTauxModel}
For a generic sequence of matrices $(A_k)_k$ of $SL_n(\R)$ (resp. $O_n(\R)$), we have
\[\tau^k(A_1,\cdots,A_k) \overset{\text{def.}}{=} D_d(\Gamma_k) = D_c\left( W^k + \widetilde\Lambda_k \right) \overset{\text{def.}}{=} \overline\tau^k(A_1,\cdots,A_k).\]
\end{prop}

Remark that the density on the left of the equality is the density of a discrete set (that is, with respect to counting measure), whereas the density on the right of the equality is that of a continuous set (that is, with respect to Lebesgue measure). The two notions coincide when we consider discrete sets as sums of Dirac masses.

A more precise version of this proposition will be given by Lemma~\ref{LemTauxDiff2}. In particular, it will imply that in the previous proposition, the generic sequence can be replaced by a sequence of iid random draws of law a probability measure whose support's interior is nonempty.

\begin{coro}\label{conttaukk}
For a generic sequence of matrices, the rate of injectivity $\tau^k$ in time $k$ coincides with the mean rate of injectivity $\overline\tau^k$, which is continuous and piecewise polynomial of degree $\le nk$ in the coefficients of the matrix.
\end{coro}

Thus, the formula of Proposition~\ref{CalculTauxModel} could be used to compute numerically the mean rate of injectivity in time $k$ of a sequence of matrices: it is much faster to compute the volume of a finite number of intersections of cubes (in fact, a small number) than to compute the cardinalities of the images of a big set $[-R,R]^n \cap \Z^n$.

\begin{proof}[Proof of Proposition \ref{CalculTauxModel}]
We want to determine the density of $\Gamma_k$. By Equation~\eqref{CalcGamma}, we have
\[x\in\Gamma_k \iff x\in\Z^n\ \text{and}\ \exists \lambda\in \Lambda_k : x=p_2(\lambda),\, p_1(\lambda) \in W^k.\]
But if $p_2(\lambda)=x$, then we can write $\lambda=(\widetilde\lambda,0^n) + (0^{(k-1)n},-x,x)$ with $\widetilde\lambda\in \widetilde\Lambda_k$. Thus,
\begin{align*}
x\in\Gamma_k & \iff x\in\Z^n\ \text{and}\ \exists \widetilde\lambda\in \widetilde\Lambda_k : (0^{(k-1)n},-x)-\widetilde\lambda	\in W^k\\
             & \iff x\in\Z^n\ \text{and}\ (0^{(k-1)n},x)\in \bigcup_{\widetilde\lambda\in\widetilde\Lambda_k} \widetilde\lambda - W^k.
\end{align*}
Thus, $x\in\Gamma_k$ if and only if the projection of $(0^{(k-1)n},x)$ on $\R^{nk}/\widetilde\Lambda_k$ belongs to $\widetilde\Lambda_k - W^k$. Then, the proposition follows directly from the fact that, under generic conditions, the points of the form $(0^{(k-1)n},x)$, with $x\in\Z^n$, are equidistributed in $\R^{nk} / \widetilde \Lambda_k$.
\bigskip

To prove this equidistribution, we compute the inverse matrix of $\widetilde M_{A_1,\cdots,A_k}$:
\[{\widetilde M_{A_1,\cdots,A_k}}^{-1} = \begin{pmatrix}
A_1^{-1} & A_1^{-1}A_2^{-1} &  A_1^{-1}A_2^{-1}A_3^{-1} & \cdots & A_1^{-1}\cdots A_k^{-1}\\
    & A_2^{-1}  & A_2^{-1}A_3^{-1} & \cdots & A_2^{-1}\cdots A_k^{-1}\\
    &      & \ddots &   & \vdots\\
    &      &        & A_{k-1}^{-1} & A_{k-1}^{-1}A_k^{-1}\\
    &     &        &          & A_k^{-1}
\end{pmatrix}.\]
Thus, the set of points of the form $(0^{(k-1)n},x)$ in $\R^{nk} / \widetilde \Lambda_k$ corresponds to the image of the action
\begin{equation}\label{EqAction}
\Z^n \ni x \longmapsto
\begin{pmatrix}
A_1^{-1}\cdots A_k^{-1}\\
A_2^{-1}\cdots A_k^{-1}\\
\vdots\\
A_{k-1}^{-1}A_k^{-1}\\
A_k^{-1}
\end{pmatrix}x
\end{equation}
of $\Z^n$ on the canonical torus $\R^{nk}/\Z^{nk}$. But this action is ergodic (even in restriction to the first coordinate) when the sequence of matrices is generic among $SL_n(\R)$. Indeed, suppose that the restriction of this action to the first coordinate (i.e. to $\Z\times 0^{n-1}$) is ergodic in restriction to the $n(k-1)$ last coordinates. It suffices to consider a matrix $A_1$ such that the entries of the first column of $A_1^{-1}$ are algebraically independent over the extension of $\Q$ generated by the coefficients of $A_2^{-1},\cdots,A_k^{-1}$. One concludes by induction on $k$.

For the case of $O_n(\R)$, the same reasoning holds under the restriction that the Euclidean norm of the first column of $A_1^{-1}$ is one. But this follows from the easy following fact: if $F\subset \R$ is a field and if $x_1,\dots,x_{n-1}\in \R$ are algebraically independent over $F$ and satisfy $x_1^2 + \cdots + x_{n-1}^2 \le 1$, then there is no nontrivial linear Diophantine equation with coefficients in $F$ satisfied by $x_1,\cdots,x_n$, where $x_n$ is such that $x_1^2 + \cdots + x_{n}^2 = 1$.
\end{proof}

As an application, a small geometric computation leads to an explanation of the figures \cite[Figure 2.3]{thibault:tel-00596947} and \cite[Figure 3.9]{ThesePluta}.

\begin{ex}
For $\theta\in[0,\pi/2]$, the mean rate of injectivity of a rotation of $\R^2$ of angle $\theta$ is (see \cite[Application 8.14.]{Guih-These}).
\[\overline\tau(R_\theta) = 1 - (\cos(\theta) + \sin(\theta) - 1)^2.\]
\end{ex}

Recall the problem raised by Theorem~\ref{ConjPrincipIntro}: we want to make $\tau^k$ tend to 0 as $k$ tends to infinity. By the equidistribution argument stated in Proposition~\ref{CalculTauxModel}, generically, it is equivalent to make the mean rate of injectivity $\overline\tau^k$ tend to 0 when $k$ goes to infinity, by perturbing every matrix in $SL_n(\R)$ of at most $\delta>0$ (fixed once for all). The conclusion of Theorem~\ref{ConjPrincipIntro} is motivated by the phenomenon of concentration of the measure on a neighbourhood of the boundary of the cubes in high dimension.

\begin{rem}\label{concentration}
Let $W^k = ]-1/2,1/2]^k\in\R^k$ and $v^k$ the vector $(1,\cdots,1)\in \R^k$. Then, for every $\varep,\delta>0$, there exists $k_0\in\N^*$ such that for every $k\ge k_0$, we have $\Leb\big(W^k \cap (W^k + \delta v^k)\big) < \varep$.
\end{rem}

The case of equality $\overline\tau^k=1$ is given by Haj\'os theorem.

\begin{theoreme}[Haj\'os, \cite{MR0006425}]\label{hajos}
Let $\Lambda$ be a lattice of $\R^n$. Then the collection of squares $\{B_\infty(\lambda,1/2)\}_{\lambda\in\Lambda}$ tiles $\R^n$ if and only if in a canonical basis of $\R^n$ (that is, permuting coordinates if necessary), $\Lambda$ admits a generating matrix which is upper triangular with ones on the diagonal.
\end{theoreme}

\begin{rem}
The kind of questions addressed by Haj\'os theorem are in general quite delicate. For example, one can wonder what happens without the assumption that the centres of the cubes form a lattice of $\R^n$. O. H. Keller conjectured in \cite{zbMATH02567416} that the conclusion of Haj\'os theorem is still true under this weaker hypothesis. This conjecture was proven to be true for $n\le 6$ by O. Perron in \cite{MR0003041,MR0002185}, but remained open in higher dimension until 1992, when J. C. Lagarias and P. W. Shor proved in \cite{MR1155280} that Keller's conjecture is false for $n\ge 10$ (this result was later improved by \cite{MR1920144} which shows that it is false as soon as $n\ge 8$; the case $n=7$ is to our knowledge still open).
\end{rem}

Combining Haj\'os theorem with Proposition~\ref{CalculTauxModel}, we obtain that the equality $\overline\tau^k=1$ occurs if and only if the lattice given by the matrix $M_{A_1,\cdots,A_k}$ satisfies the conclusions of Haj\'os theorem\footnote{Of course, this property can be obtained directly by saying that the density is equal to 1 if and only if the rate of injectivity of every matrix of the sequence is equal to 1.}. The heuristic suggested by the phenomenon of concentration of the measure is that if we perturb ``randomly'' any sequence of matrices, we will go ``far away'' from the lattices satisfying Haj{\'o}s theorem and then the rate of injectivity will be close to 0.
\bigskip

In fact, the proof of Proposition \ref{CalculTauxModel} shows a more precise statement: if the image of \eqref{EqAction} is $\delta$-dense, then $\tau^k(A_1,\cdots,A_k)$ is close to $\overline\tau^k(A_1,\cdots,A_k)$. This leads to the following statement.

\begin{prop}\label{PropMemMat}
Let $n\ge 2$ and $G$ be one of the groups $GL_n(\R)$, $SL_n(\R)$ or $SO_n(\R)$. Then there exists a countable collection of positive codimension submanifolds $V_q$ of $G$ such that for any $\varep>0$ and any time $k\in\N$, there exists $M\in\N$ such that if $A\in G\setminus \bigcup_{q=1}^M V_q$, then 
\[\Big| \tau^k(A,\cdots,A) - \overline\tau^k(A,\cdots,A)\Big| \le\varep.\]
In particular, there exists a generic full measure subset of $G$ on which $\tau^k = \overline\tau^k$ for any $k\in\N$.
\end{prop}

\begin{proof}[Proof of Proposition \ref{PropMemMat}]
By using Weyl's criterion (Proposition \ref{Weyl}) and the proof of Proposition \ref{CalculTauxModel}, it suffices to prove that for any $\varep>0$ and any $k\in\N$, there exists $C\in\N$ and a locally finite collection of positive codimension submanifolds of $G^{-1} = G$ on which at least one column of the matrix
\[\begin{pmatrix}
A^{k}\\
A^{k-1}\\
\vdots\\
A^{2}\\
A
\end{pmatrix}\]
forms an ``almost $\Q$-free family'', that is, is not a solution of any linear Diophantine equation with coefficients in $\{-C,\cdots,C\}$. Let $V$ be the set of solutions of such an equation; as this equation is polynomial in the coefficients of $A$ this is a subvariety of $GL_n(\R)$. For $G = GL_n(\R)$, this is also a positive codimension subvariety of $G$.

For the other cases, $G$ is an irreducible variety; thus to prove that $V\cap G$ is a positive codimension subvariety of $G$ it suffices to prove that the set $G\setminus V$ is nonempty (see for instance \cite[Chap. III prop. 7]{MR0217337}).

If $G$ contains both the sets of diagonal and permutation matrices, it suffices to  consider a diagonal matrix with a transcendent entry and conjugate it by a proper permutation matrix.

If $G = SO_n(\R)$, Lindemann-Weierstrass theorem implies that for $\theta\in\Q$, one has no nontrivial linear Diophantine equation with rational coefficients satisfied by the numbers $\cos(k\theta)$ and $\sin(k\theta)$: if it was the case, $e^{i\theta}$ would be a zero of a polynomial in $(\Q(i))[X]$, which would contradict the fact that $\big[\Q(e^{i\theta}):\Q\big]=\infty$. We then consider the matrix given by the diagonal blocks $R_\theta$ and $I_{n-2}$ and conjugate it by a proper permutation matrix; this gives us a matrix of $SO_n(\R)$ which is not in $V$.
\end{proof}

\subsection[Proof that the rate is smaller than $1/2$]{A first step: proof that the asymptotic rate of injectivity is generically smaller than $1/2$}

As a first step, we prove that rate of injectivity of a generic sequence of $\ell^\infty(SL_n(\R))$ is smaller than $1/2$.

\begin{prop}\label{PerLin1}
There exists an open and dense subset of sequences of matrices in $SL_n(\R)$ in which every sequence $(A_k)_{k\ge 1}$ satisfies: there exists a parameter $\alpha\in ]0,1[$ such that for every $k\ge 1$, we have $\overline\tau^k(A_1,\cdots,A_k) \le (\alpha^k+1)/2$. In particular, $\overline\tau^\infty((A_k)_k)\le 1/2$.
\end{prop}

To begin with, we give a lemma estimating the sizes of intersections of cubes when the mean rate of injectivity $\overline\tau^k$ is bigger than $1/2$.

\begin{lemme}\label{EstimTaux}
Let $W = ]-1/2,1/2]^m$ and $\Lambda\subset \R^m$ be a lattice with covolume 1 such that $D_c(W + \Lambda)\ge 1/2$. Then, for every $v\in\R^m$, we have
\[D_c\big((W + \Lambda + v)\cap (W + \Lambda)\big) \ge 2D_c(W + \Lambda)-1.\]
\end{lemme}

\begin{proof}[Proof of Lemma~\ref{EstimTaux}]
We first remark that $D_c(\Lambda+W)$ is equal to the volume of the projection of $W$ on the quotient space $\R^m/\Lambda$. For every $v\in\R^m$, the projection of $W+v$ on $\R^m/\Lambda$ has the same volume; as this volume is greater than $1/2$, and as the covolume of $\Lambda$ is 1, the projections of $W$ and $W+v$ overlap, and the volume of the intersection is bigger than $2D_c(W + \Lambda)-1$. Returning to the whole space $\R^m$, we get the conclusion of the lemma.
\end{proof}

A simple counting argument leads to the proof of the following lemma.

\begin{lemme}\label{DoubleComptage}
Let $\Lambda_1$ be a subgroup of $\R^m$, $\Lambda_2$ be such that $\Lambda_1 \oplus \Lambda_2$ is a lattice of covolume 1 of $\R^m$, and $C$ be a compact subset of $\R^m$. Let $C_1$ be the projection of $C$ on the quotient $\R^m/\Lambda_1$, and $C_2$ the projection of $C$ on the quotient $\R^m/(\Lambda_1\oplus\Lambda_2)$. We denote by
\[a_i = \Leb\big\{x\in C_1 \mid \card\{\lambda_2\in\Lambda_2 \mid x\in C_1+\lambda_2\} = i \big\}\]
(in particular, $\sum_{i\ge 1} a_i = \Leb(C_1)$). Then,
\[\Leb(C_2) = \sum_{i\ge 1} \frac{a_i}{i}.\]
\end{lemme}

In particular, the area of $C_2$ (the projection on the quotient by $\Lambda_1 \oplus \Lambda_2$) is smaller than (or equal to) that of $C_1$ (the projection on the quotient by $\Lambda_1$). The loss of area is given by the following corollary.

\begin{coro}\label{CoroSansNom}
With the same notations as for Lemma~\ref{DoubleComptage}, if we denote by
\[\widetilde C_1 = \Leb\big\{x\in C_1 \mid \card\{\lambda_2\in\Lambda_2 \mid x\in C_1+\lambda_2\} \ge 2 \big\},\]
then,
\[\Leb(C_2) \le \Leb(C_1) - \frac{\widetilde C_1}{2}.\]
\end{coro}

\begin{proof}[Proof of Proposition~\ref{PerLin1}]\label{Proof2PerLin1}
Let $\delta>0$ and $M>0$. We proceed by induction on $k$ and suppose that the following property is proved until a certain rank $k\in\N^*$: there exists $\alpha>0$, depending only on $\delta$ and $M$, such that for any sequence of matrices $(A_k)_{k\ge 1}$ of $SL_n(\R)$ with $\|(A_k)_k\|_\infty \le M$, there exists $B_1,\dots,B_k \in SL_n(\R)$ such that $\|A_j - B_j\| \le \delta$ for any $j$ and $\overline\tau^k(B_1,\cdots,B_k) \le (\alpha^k+1)/2$.

Let $\widetilde \Lambda_k$ be the lattice spanned by the matrix $\widetilde M_{B_1,\cdots,B_k}$ and $W^k = ]-1/2, 1/2]^{nk}$ be the window corresponding to the model set $\Gamma_k$ modeled on $\Lambda_k$ (the lattice spanned by the matrix $M_{B_1,\cdots,B_k}$, see Equations~\eqref{DefMat1} and \eqref{DefMat2}).

We now choose a matrix $B_{k+1}$ satisfying $\|A_{k+1} - B_{k+1}\|\le\delta$, such that there exists $x_1\in \Z^n\setminus \{0\}$ such that $\| B_{k+1}x_1\|_\infty \le 1-\varep$, with $\varep>0$ depending only on $\delta$ and $M$ (and $n$): indeed, for every matrix $B\in SL_n(\R)$, Minkowski theorem implies that there exists $x_1\in \Z^n\setminus \{0\}$ such that $\| Bx_1\|_\infty \le 1$; it then suffices to modify slightly $B$ to decrease $\| Bx_1\|_\infty$. By the form of the matrix $\widetilde M_{B_1,\cdots,B_{k+1}}$, we have the decomposition
\[W^{k+1}+ \widetilde \Lambda_{k+1} = W^{k+1} + \begin{pmatrix} \widetilde \Lambda_k\\0^n \end{pmatrix} +
\begin{pmatrix} 0^{n(k-1)}\\ -I_n\\ B_{k+1} \end{pmatrix} \Z^n.\]
In particular, as $|\det(B_{k+1})|=1$, this easily implies that
\[\overline\tau^{k+1} \overset{\text{def.}}{=} D_c\left(W^{k+1}+ \widetilde \Lambda_{k+1}\right)\le D_c\left(W^k+ \widetilde \Lambda_{k} \right) \overset{\text{def.}}{=} \overline\tau^k.\]

\begin{figure}[t]
\begin{minipage}[t]{\linewidth}
\emph{How to read these figures :} The top of the figure represents the set $W^k + \widetilde\Lambda_k$ by the 1-dimensional set $[-1/2,1/2] + \nu\Z$ (in dark blue), for a number $\nu>1$. The bottom of the figure represents the set $W^{k+1} + \widetilde\Lambda_{k+1}\subset \R^{n(k+1)}$ by the set $[-1/2,1/2]^2 + \Lambda$, where $\Lambda$ is the lattice of $\R^2$ spanned by the vectors $(0,\nu)$ and $(1,1-\varep)$ for a parameter $\varep>0$ close to 0. The dark blue cubes represent the ``old'' cubes, that is, the thickening $W^{k+1} + (\widetilde\Lambda_k,0^n)$ of the set $W^k + \widetilde\Lambda_k$, and the light blue cubes represent the ``added'' cubes, that is, the rest of the set $W^{k+1} + \widetilde\Lambda_{k+1}$.
\end{minipage}\vspace{25pt}

\begin{minipage}[t]{.48\linewidth}
\begin{center}
\begin{tikzpicture}[scale=.95]
\draw[color=gray] (-.8,3) -- (5,3);
\draw[very thick,blue, |-|] (-.5,3) -- (.5,3);
\draw[very thick,blue, |-|] (-.5+1.8,3) -- (.5+1.8,3);
\draw[very thick,blue, |-|] (-.5+3.6,3) -- (.5+3.6,3);

\clip (-.6,-1.3) rectangle (4.8,1.5);
\foreach\j in {0,...,2}{
\draw[fill=blue,opacity=.25] (-.5+1.8*\j,-.5) rectangle (.5+1.8*\j,.5);
}
\foreach\i in {-2,...,2}{
\foreach\j in {-1,...,3}{
\draw[fill=blue,opacity=.20] (-.5+\i+1.8*\j,-.5+.8*\i) rectangle (.5+\i+1.8*\j,.5+.8*\i);
\draw (-.5+\i+1.8*\j,-.5+.8*\i) rectangle (.5+\i+1.8*\j,.5+.8*\i);
}}
\draw[dashed] (-1,-.4) -- (5,-.4);
\fill[white, opacity=0.5] (1.2,-.4) rectangle (1.5,-.7);
\draw (1.35,-.55) node {$v$};
\draw[->,>=latex] (1.8,0) -- (0.8,-0.8);
\fill (0.3,-0.3) rectangle (0.5,-0.5);
\end{tikzpicture}
\caption[Intersection of cubes, rate bigger than $1/2$]{In the case where the rate is bigger than $1/2$, some intersections of cubes appear automatically between times $k$ and $k+1$.}\label{FigTauDemi}
\end{center}
\end{minipage}\hfill
\begin{minipage}[t]{.48\linewidth}
\begin{center}
\begin{tikzpicture}[scale=.95]
\draw[color=gray] (-.8,3) -- (5,3);
\draw[very thick,blue, |-|] (-.5,3) -- (.5,3);
\draw[very thick,blue, |-|] (-.5+2.8,3) -- (.5+2.8,3);

\clip (-.6,-1.3) rectangle (4.8,1.5);
\foreach\j in {0,...,1}{
\draw[fill=blue,opacity=.25] (-.5+2.8*\j,-.5) rectangle (.5+2.8*\j,.5);
}
\foreach\i in {-3,...,2}{
\foreach\j in {-1,...,2}{
\draw[fill=blue,opacity=.20] (-.5+\i+2.8*\j,-.5+.8*\i) rectangle (.5+\i+2.8*\j,.5+.8*\i);
\draw (-.5+\i+2.8*\j,-.5+.8*\i) rectangle (.5+\i+2.8*\j,.5+.8*\i);
}}
\end{tikzpicture}
\caption[Intersection of cubes, rate smaller than $1/2$]{In the case where the rate is smaller than $1/2$, there is not necessarily new intersections between times $k$ and $k+1$.}\label{FigTauTiers}
\end{center}
\end{minipage}
\end{figure}

What we need is a more precise bound. We apply Corollary~\ref{CoroSansNom} to
\[\Lambda_1 = \big(\widetilde\Lambda_k,0^n\big), \qquad \Lambda_2 = \begin{pmatrix} 0^{n(k-1)}\\ -I_n\\ B_{k+1} \end{pmatrix} \Z^n \qquad \text{and} \qquad C = W^{k+1}.\]
Then, the decreasing of the rate of injectivity between times $k$ and $k+1$ is bigger than the $\widetilde C_1$ defined in Corollary~\ref{CoroSansNom}: using Lemma~\ref{EstimTaux}, we have
\[D_c\left(\Big(W^k + (\widetilde\Lambda_k,0^n)\Big) \cap \Big(W^k + (\widetilde\Lambda_k,0^n) + \big(0^{n(k-1)},-x_1\big)\Big)\right) \ge 2D_c\left(W^k + \widetilde\Lambda_k\right)-1.\]
On Figure \ref{FigTauDemi}, this corresponds to the restriction to the horizontal dashed line: on this subspace one can see the sets $W^{k+1} + (\widetilde\Lambda_k,0^n)$ (in dark blue) and $W^{k+1} + (\widetilde\Lambda_k,0^n) + v$. The previous formula gives a lower bound of the horizontal size of the black rectangle. As $\|x_1\|_\infty <1-\varep$, the vertical size of this rectangle is bigger than $\varep^n$, which leads to a total area of the rectangle:
\[\widetilde C_1 \ge \varep^n \left(2D_c\left(W^k + \widetilde\Lambda_k\right)-1\right).\]
From Corollary~\ref{CoroSansNom} we deduce that
\[D_c\left(W^{k+1}+ \widetilde \Lambda_{k+1}\right)
 \le D_c\left(W^k + \widetilde\Lambda_k\right) - \frac12 \widetilde C_1,
\]
Hence,
\[\overline\tau^{k+1}
\le \overline\tau^k - \frac12 \varep^n \big(2\overline\tau^k-1\big).
\]
This proves the property for rank $k+1$, by setting $\alpha = 1-\varep^n$.

Finally, we have proved the following: for every $\delta>0$ and any $M>0$, there exists a number $\alpha<1$ and a $\delta$-dense open subset of sequences of matrices of $SL_n(\R)$ of norm smaller than $M$ on which $\overline\tau^k(B_1,\cdots,B_k) \le (\alpha^k+1)/2$. This gives the proposition.
\end{proof}

\subsection[Generically, the asymptotic rate is zero]{Proof of Theorem~\ref{ConjPrincipIntro}: generically, the asymptotic rate is zero}

We now come to the proof of Theorem~\ref{ConjPrincipIntro}. The strategy of proof is identical to that we used in the previous section to state that generically, the asymptotic rate is smaller than $1/2$ (Proposition~\ref{PerLin1}): we will use an induction to decrease the rate step by step. Recall that $\overline\tau^k(A_1,\cdots,A_k)$ indicates the density of the set $W^{k+1} + \widetilde\Lambda_k$.

Unfortunately, if the density of $W^k + \widetilde\Lambda_k$ --- which is generically equal to the density of the $k$-th image $\big(\widehat A_k\circ\cdots\circ \widehat A_1\big)(\Z^n)$ --- is smaller than $1/2$, then we cannot apply exactly the strategy of proof of the previous section (see Figure~\ref{FigTauTiers}). For example, taking
\[A_1 = \operatorname{diag}(100,1/100)  \quad  \text{and}\quad A_2 = \operatorname{diag}(1/10,10),\]
one gets $\big(\widehat{A_2}\circ \widehat{A_1}\big)(\Z^2) = (10\Z)^2$, and for every $B_3$ close to the identity, we have $\tau^3(A_1,A_2,B_3) = \tau^2(A_1,A_2) = 1/100$. This example tells us that there exists an open set of $\ell^\infty(SL_n(\R))$ on which the sequence $\tau^k$ is not strictly decreasing.

Moreover, if we set $A_k=\operatorname{Id}$ for every $k\ge 2$, then we can set $B_1=A_1$, $B_2=A_2$ and for each $k\ge 2$ perturb each matrix $A_k$ into the matrix $B_k = \operatorname{diag}(1+\delta,1/(1+\delta))$, with $\delta>0$ small. In this case, there exists a time $k_0$ (minimal) such that $\tau^{k_0}(B_1,\cdots,B_{k_0}) < 1/100$. But if instead of $A_{k_0} = \operatorname{Id}$, we had $A_{k_0} = \operatorname{diag}(1/5,5)$, this construction would not work anymore: we should have set $B_k = \operatorname{diag}(1/(1+\delta),1+\delta)$ instead. This suggests that we should take into account the next terms of the sequence $(A_k)_k$ to perform the perturbations. And things seems even more complicated when the matrices are no longer diagonal\dots
\bigskip

To overcome this difficulty, we use the same proof strategy than in the case where $\overline\tau\ge 1/2$, provided we wait long enough: for a generic sequence of matrices $(A_k)_{k\ge 1}$, if $\overline\tau^k(A_1,\cdots,A_k)> 1/\ell$, then $\overline\tau^{k+\ell-1}(A_1,\cdots,A_{k+\ell-1})$ is strictly smaller than $\overline\tau^k(A_1,\cdots,A_k)$. More precisely, we consider the maximal number of disjoint translates of $W^k + \widetilde\Lambda_k$ in $\R^{nk}$: generalizing Lemma \ref{EstimTaux}, we easily see that if the density of $W^k + \widetilde\Lambda_k$ is bigger than $1/\ell$, then there cannot be more than $\ell$ disjoint translates of $W^k + \widetilde\Lambda_k$ in $\R^{nk}$(Lemma~\ref{EstimTauxN}). At this point, Lemma~\ref{LemDeFin} will state that if the sequence of matrices is generic, then either the density of $W^{k+1} + \widetilde\Lambda_{k+1}$ is smaller than that of $W^k + \widetilde\Lambda_k$ (Figure~\ref{FigLemDeFin}), or there cannot be more than $\ell-1$ disjoint translates of $W^{k+1} + \widetilde\Lambda_{k+1}$ in $\R^{n(k+1)}$(Figure~\ref{FigLemDeFin2}). Applying this reasoning (at most) $\ell-1$ times, we obtain that the density of $W^{k+\ell-1} + \widetilde\Lambda_{k+\ell-1}$ is smaller than that of $W^k + \widetilde\Lambda_k$. For example if $D_c\big(W^k + \widetilde\Lambda_k\big) > 1/3$, then $D_c\big(W^{k+2} + \widetilde\Lambda_{k+2}\big) < D\big(W^k + \widetilde\Lambda_k\big)$ (Figure~\ref{FigInterCubes3DSupDemi}). To apply this strategy in practice, we have to obtain quantitative estimates about the loss of density we get between times $k$ and $k+\ell-1$.

Remark that with this strategy we do not need to make ``clever'' perturbations of the matrices: provided that the coefficients of the matrices are rationally independent, the perturbation of each matrix is made independently from that of the others. However, this reasoning does not tell us exactly when the rate of injectivity decreases (likely, in most of cases, the speed of decreasing of the rate of injectivity is much faster than the one obtained by this method), and does not say either where exactly the loss of injectivity occurs in the image sets.

We will indeed prove a more precise statement of Theorem~\ref{ConjPrincipIntro}.

\begin{theoreme}\label{ConjPrincip}
For a generic sequence of matrices $(A_k)_{k\ge 1}$ of $\ell^\infty(SL_n(\R))$,
\[\overline\tau^\infty\big( (A_k)_{k\ge 1}\big)=0,\]
and for every $\varep>0$, the set of $(A_k)_{k\ge 1}\in\ell^\infty(SL_n(\R))$ such that $\overline\tau^\infty\big( (A_k)_{k\ge 1}\big) <\varep$ is open and dense.

Moreover, for any $\ell>0$, there exists a locally finite number of positive codimension submanifolds $V_q$ of $SL_n(\R)$ such that for any $\delta>0$, if $(A_k)_{k\ge 1} \in \ell^\infty(SL_n(\R))$ is such that $d\big(A_k,\bigcup_q V_q\big)>\delta$ for a positive asymptotic density of integers $k$, then $\overline\tau^\infty\big( (A_k)_{k\ge 1}\big) \le 1/\ell$.
\end{theoreme}

The existence of open and dense sets of sequences on which $\overline\tau^\infty$ is smaller than $\varep$ is easily deduced from the first part of the theorem by applying the continuity of $\overline\tau^k$. By using the continuity of $\tau$ on a generic subset (Corollary~\ref{conttaukk}), one gets the same statement for $\tau$.

\begin{rem}
It is possible to prove that the intersection of the submanifold $V_q$ with $O_n(\R)$ is again a positive codimension submanifold of $O_n(\R)$. From this fact we deduce the fact announced in the introduction that Theorem \ref{ConjPrincip} remains true for generic sequences of isometries (see also \cite{Gui15b} for an alternative proof).
\end{rem}

The last part of this theorem deals with $\overline\tau$; one can easily get a similar statement for $\tau$ by combining it with Lemma \ref{LemTauxDiff2}. This leads to the following consequence for sequences of random iid matrices.

\begin{coro}\label{CoroiidMat}
Let $\mu$ be a measure on $SL_n(\R)$ with nonempty interior support. Then, for almost every sequence $(A_k)_{k\ge 1}$ of independent identically distributed matrices with distribution $\mu$, we have $\tau^\infty\big( (A_k)_{k\ge 1}\big)=0$.
\end{coro}

By applying Proposition \ref{PropMemMat}, one gets the following corollary.

\begin{coro}\label{CoroMemMat}
Let $n\ge 2$ and $G$ be one of the groups $GL_n(\R)$, $SL_n(\R)$ or $SO_n(\R)$. Then, for a generic/almost all $A\in G$, one has $\tau^\infty\big( (A)\big)=0$.
\end{coro}


We now come to the proof of Theorem \ref{ConjPrincip}. The following lemma is a generalization of Lemma~\ref{EstimTaux}. It expresses that if the density of $W^k + \widetilde\Lambda_k$ is bigger than $1/\ell$, then there cannot be more than $\ell$ disjoint translates of $W^k + \widetilde\Lambda_k$, and gives an estimation on the size of these intersections.

\begin{lemme}\label{EstimTauxN}
Let $W = ]-1/2,1/2]^m$ and $\Lambda\subset \R^m$ be a lattice with covolume 1 such that $D_c(W + \Lambda)\ge 1/\ell$. Then, for every collection $v_1,\cdots,v_\ell\in\R^m$, there exists $i\neq i'\in \llbracket 1,\ell\rrbracket$ such that
\[D_c\big((W + \Lambda + v_i)\cap (W + \Lambda + v_{i'})\big) \ge \frac{\ell D_c(W + \Lambda)-1}{\ell(\ell-1)}.\]
\end{lemme}

\begin{proof}[Proof of Lemma~\ref{EstimTauxN}]
For every $v\in\R^m$, the density $D_c(W + \Lambda+v)$ is equal to the volume of the projection of $W$ on the quotient space $\R^m/\Lambda$. As this volume is greater than $1/\ell$, and as the covolume of $\Lambda$ is 1, the projections of the $W + v_i$ overlap. Quantitatively, using Bonferroni inequality,
\[\sum_i\Leb_\Lambda\big(W + v_i\big) \le \Leb_\Lambda\Big(\bigcup_i (W + v_i)\Big) + \sum_{i\neq j} \Leb_\Lambda\big((W + v_i)\cap (W + v_j)\big),\]
one gets that there exists $i\neq j$ such that 
\[\Leb_\Lambda\big((W + v_i)\cap (W + v_j)\big) \ge \frac{\ell\Leb_\Lambda(W)-1}{\ell(\ell-1)}.\]
Returning to the whole space $\R^m$, we get the conclusion of the lemma.
\end{proof}

Proof of Theorem~\ref{ConjPrincip} will reduce to the following technical lemma.

\begin{lemme}\label{LemDeFin}
Fix $M>0$ and $L\ge 2$. Then there exists a finite collection $(V_q)_q$ of positive codimension submanifolds of $SL_n(\R)$ such that for every $\delta>0$, there exists $\varep>0$ satisfying:\\
For any $k\ge 0$ and $B_1,\cdots,B_k\in SL_n(\R)$ such that there exist $\ell\le L$ and $D_k^\ell>0$ such that for every collection of vectors $v_1,\cdots,v_{\ell} \in\R^n$, there exists $j\neq j'\in\llbracket 1,\ell\rrbracket$ such that
\begin{equation}\label{Deca}
D_c\bigg(\Big(W^k + \widetilde\Lambda_k + (0^{(k-1)n} , v_j) \Big) \cap \Big(W^k + \widetilde\Lambda_k + (0^{(k-1)n} , v_{j'}) \Big)\bigg)\ge D_k^\ell,
\end{equation}
where $\widetilde\Lambda_k$ is spanned by $\widetilde M_{B_1,\cdots,B_k}$ (see \eqref{DefMat2}).\\
Then for every matrix $B\in SL_n(\R)$ of norm smaller than $M$ such that $d(B,V_q)>\delta$ for every $q$, if we denote by $\widetilde\Lambda_{k+1}$ the lattice spanned by the matrix $\widetilde M_{B_1,\cdots,B_k,B}$,
\begin{enumerate}[(1)]
\item either $D_c(W^{k+1} + \widetilde\Lambda_{k+1}) \le D_c(W^{k} + \widetilde\Lambda_{k}) - \varep D_k^\ell$;
\item or for every collection of vectors $w_1,\cdots,w_{\ell-1} \in\R^n$, there exists $i\neq i'\in\llbracket 1,\ell-1\rrbracket$ such that
\[D_c\bigg(\Big(W^{k+1} + \widetilde\Lambda_{k+1} + (0^{kn} , w_i) \Big) \cap \Big(W^{k+1} + \widetilde\Lambda_{k+1} + (0^{kn} , w_{i'}) \Big)\bigg)\ge \varep D_k^\ell.\]
\end{enumerate}
Moreover, if $\ell = 2$, then we have automatically the conclusion (1) of the lemma (as a particular case one gets Proposition~\ref{PerLin1}).
\end{lemme}

Conclusion (1) says that the rate of injectivity decreases between times $k$ and $k+1$, and conclusion (2) expresses that the number of possible disjoint translates of $W^k + \widetilde\Lambda_k$ decreases between times $k$ and $k+1$. In a certain sense, conclusion (1) corresponds to an hyperbolic behaviour, and conclusion (2) to a diffusion between times $k$ and $k+1$. The same strategy of proof leads to the following variation of the lemma, which is true without condition on the matrix $B$.

\begin{lemme}\label{LemDeFin2}
Fix $M>0$ and $\ell\ge 2$. Then there exists $\varep>0$ such that for any $k\in\N$, if $D_k^\ell>0$ is such that for every collection of vectors $v_1,\cdots,v_{\ell} \in\R^n$, there exists $j,j'\in\llbracket 1,\ell\rrbracket$ such that
\[D_c\bigg(\Big(W^k + \widetilde\Lambda_k + (0^{(k-1)n} , v_j) \Big) \cap \Big(W^k + \widetilde\Lambda_k + (0^{(k-1)n} , v_{j'}) \Big)\bigg)\ge D_k^\ell,\]
then for every matrix $B\in SL_n(\R)$ of norm smaller than $M$, if we denote by $\widetilde\Lambda_{k+1}$ the lattice spanned by the matrix $\widetilde M_{B_1,\cdots,B_k,B}$, then for every collection of vectors $w_1,\cdots,w_{\ell} \in\R^n$, there exists $i\neq i'\in\llbracket 1,\ell\rrbracket$ such that
\[D_c\bigg(\Big(W^{k+1} + \widetilde\Lambda_{k+1} + (0^{kn} , w_i) \Big) \cap \Big(W^{k+1} + \widetilde\Lambda_{k+1} + (0^{kn} , w_{i'}) \Big)\bigg)\ge \varep D_k^\ell.\]
\end{lemme}

We begin by defining the sets $V_q$. A collection of cubes of $\R^n$ is called a \emph{$k$-fold tiling}\footnote{See also Furtw{\"a}ngler conjecture \cite{MR1550530}, proved false by G.~Haj{\'o}s. R.~Robinson gave a characterization of such $k$-fold tilings in some cases, see \cite{MR526466} or \cite[p. 29]{MR1311249}.} if any point of $\R^n$ belongs to exactly $k$ different cubes, except from those belonging to the boundary of one cube.

\begin{lemme}\label{Kfold}
Let $M>0$. For any matrix $B\in SL_n(\R)$, consider the collection $(Bv + W^1)_{v\in\Z^n}$ of translates of the unit cube by the lattice spanned by $B$. Now, take $\ell-1$ vectors $w_1,\cdots,w_{\ell-1} \in\R^n$ and consider the collection
\begin{equation}\label{EqColl}
(Bv + w_i + W^1)_{v\in\Z^n,\,1\le i\le\ell-1}
\end{equation}
of translates of the first collection by the vectors $w_i$. Then, the set of matrices $B\in SL_n(\R)$ of norm $\le M$ such that the collection \eqref{EqColl} forms a $\ell-1$-fold tiling is contained in a finite union of positive codimension submanifolds of $SL_n(\R)$, which we denote by $V_q$.
\end{lemme}

\begin{center}
\begin{tikzpicture}[scale=.95]
\draw[fill=blue!15!white] (-.5,-.5) rectangle (.5,.5);
\draw[fill=blue!15!white] (.5,-.4) rectangle (1.5,.6);
\draw[fill=blue!15!white] (1.5,-.3) rectangle (2.5,.7);
\draw[fill=blue!15!white] (2.5,-.6) rectangle (3.5,.4);
\draw[color=blue!25!black] (0,.5) node[above]{$W_0$};
\draw[color=blue!25!black] (1,.6) node[above]{$W_1$};
\draw[color=blue!25!black] (2,.7) node[above]{$W_2$};
\draw[color=blue!25!black] (3,.4) node[above]{$W_3$};
\foreach\i in {0,...,3}{
\draw(\i+.5,0) node{$\times$};
}
\draw(0,0) node{$+$};
\end{tikzpicture}
\end{center}

\begin{proof}[Proof of Lemma \ref{Kfold}]
Indeed, consider a matrix $B$ and $w_1,\cdots,w_{\ell-1} \in\R^n$ such that the collection \eqref{EqColl} forms a $\ell-1$-fold tiling. Translating the whole family if necessary, one can suppose that $w_1=0$. Then, easy geometrical arguments show that for every $i$, there exists a cube $W_i$ of the collection \eqref{EqColl} such that $(-1/2+i,0,0,\dots)\in\R^n$ belongs to the left face of $W_i$ (with the convention $W_0 = W^1$). By the pigeonhole principle, two cubes of the collection $(W_i)_{0\le i \le \ell-1}$ belong to the same translated family of cubes: there exists $i_0\in\llbracket 1,\ell-1\rrbracket$, $v,v'\in\Z^n$ and $i\neq j\in\llbracket 0,\ell-1\rrbracket$ such that $W_i = Bv + w_{i_0} + W^1$ and $W_j = Bv' + w_{i_0} + W^1$. But $W_i$ and $W_j$ differ of a translation of the form $(j-i,*,*,\dots)$, thus the first coordinate of $B(v-v')$ has to be an integer smaller than $\ell$. Such a reasoning can be made on each coordinate, and the obtained conditions are about different coefficients of the matrix $B$. Hence, the set of matrices of norm $\le M$ forming a $\ell-1$-fold tiling is contained in a finite union of codimension $\ge n$ submanifolds of $M_n(\R)$, thus the restriction of it to $SL_n(\R)$ is of codimension $\ge n-1$.
\end{proof}

\begin{proof}[Proof of Lemma \ref{LemDeFin}]
Let $\delta>0$ and $M>0$. We define $\mathcal O_\varep$ as the set of matrices $B\in SL_n(\R)$ satisfying: for any $w_1,\cdots,w_{\ell-1} \in\R^n$, there exists a set $U\subset\R^n/B\Z^n$ of measure $>\varep$ such that every point of $U$ belongs to at least $\ell$ different cubes of the collection  \eqref{EqColl}. In other words, every $x\in\R^n$ whose projection $\overline x$ on $\R^n/B\Z^n$ belongs to $U$ satisfies
\begin{equation}\label{EqLemDeFin}
\sum_{i=1}^{\ell-1} \sum_{v\in \Z^n}  \1_{x\in Bv + w_i + W^1} \ge \ell.
\end{equation}
By definition, the (monotonic) union of $\mathcal O_\varep$ over $\varep>0$ is the complement of the union of the submanifolds $V_q$. By compactness of the set of matrices of $SL_n(\R)$ of norm smaller than $M$, and continuity of the map $B\mapsto \Leb(U)$, this implies the existence of $\varep = \varep(\delta,M,\ell)>0$ such that $\mathcal O = \mathcal O_\varep$ is contained in the complement of the union over $q$ of the $\delta$-neighbourhoods of $V_q$.

We then choose $B\in \mathcal O$ and a collection of vectors $w_1,\cdots,w_{\ell-1} \in\R^n$. Let $x\in\R^n$ be such that $\overline x\in U$. By hypothesis on the matrix $B$, $x$ satisfies Equation~\eqref{EqLemDeFin}, so there exists $\ell+1$ integer vectors $v_1,\cdots,v_{\ell}$ and $\ell$ indices $i_1,\cdots,i_{\ell}$ such that the couples $(v_j,i_j)$ are pairwise distinct and that
\begin{equation}\label{EqIxIn}
\forall j\in\llbracket 1,\ell\rrbracket,\quad x\in W^1 + Bv_j + w_{i_j}.
\end{equation}

\begin{figure}
\begin{minipage}[t]{.48\linewidth}
\begin{center}
\begin{tikzpicture}[scale=.95]
\draw[color=gray] (-.8,3) -- (5,3);
\draw[very thick,blue, |-|] (-.5,3) -- (.5,3);
\draw[very thick,blue, |-|] (-.5+2.8,3) -- (.5+2.8,3);

\draw[dashed] (-.8,.4) -- (5,.4);
\draw(-.8,.4) node[left]{$x$};
\draw[dashed] (2.4,-1.5) -- (2.4,1.7);
\draw (2.4,-1.5) node[below]{$y$};

\clip (-.6,-1.3) rectangle (4.8,1.5);
\foreach\j in {0,...,1}{
\draw[fill=blue,opacity=.25] (-.5+2.8*\j,-.5) rectangle (.5+2.8*\j,.5);
}
\foreach\i in {-3,...,5}{
\foreach\j in {-1,...,3}{
\draw[fill=blue,opacity=.20] (-.5+\i-1+2.8*\j,-.5+.4*\i-.4) rectangle (.5+\i-1+2.8*\j,.5+.4*\i-.4);
\draw (-.5+\i-1+2.8*\j,-.5+.4*\i-.4) rectangle (.5+\i-1+2.8*\j,.5+.4*\i-.4);
}}
\end{tikzpicture}
\caption[First case of Lemma \ref{LemDeFin}]{First case of Lemma \ref{LemDeFin}, in the case $\ell=3$: the set $W^{k+1} + \widetilde\Lambda_{k+1}$ auto-intersects.}\label{FigLemDeFin}
\end{center}
\end{minipage}\hfill
\begin{minipage}[t]{.48\linewidth}
\begin{center}
\begin{tikzpicture}[scale=.95]
\draw[color=gray] (-.8,3) -- (5,3);
\draw[very thick,blue, |-|] (-.5,3) -- (.5,3);
\draw[very thick,blue, |-|] (-.5+2.8,3) -- (.5+2.8,3);

\draw[dashed] (-.8,.4) -- (5,.4);
\draw(-.8,.4) node[left]{$x$};
\draw[dashed] (1.4,-1.5) -- (1.4,1.7);
\draw (1.4,-1.5) node[below]{$y$};

\clip (-.6,-1.3) rectangle (4.8,1.5);
\foreach\j in {0,...,1}{
\draw[fill=blue,opacity=.25] (-.5+2.8*\j,-.5) rectangle (.5+2.8*\j,.5);
}
\foreach\i in {-3,...,2}{
\foreach\j in {-1,...,2}{
\draw[fill=blue,opacity=.20] (-.5+\i+2.8*\j,-.5+.8*\i) rectangle (.5+\i+2.8*\j,.5+.8*\i);
\draw (-.5+\i+2.8*\j,-.5+.8*\i) rectangle (.5+\i+2.8*\j,.5+.8*\i);
\draw[fill=gray,opacity=.15] (-.5+\i+2.8*\j,-.5+.8*\i+1.2) rectangle (.5+\i+2.8*\j,.5+.8*\i+1.2);
\draw (-.5+\i+2.8*\j,-.5+.8*\i+1.2) rectangle (.5+\i+2.8*\j,.5+.8*\i+1.2);
}}
\end{tikzpicture}
\caption[Second case of Lemma \ref{LemDeFin}]{Second case of Lemma \ref{LemDeFin}, in the case $\ell=3$: two distinct vertical translates of $W^{k+1} + \widetilde\Lambda_{k+1}$ intersect (the first translate contains the dark blue thickening of $W^k + \widetilde\Lambda_k$, the second is represented in grey).}\label{FigLemDeFin2}
\end{center}
\end{minipage}
\end{figure}

The following formula makes the link between what happens in the $n$ last and in the $n$ penultimates coordinates of $\R^{n(k+1)}$:
\begin{equation}\label{EqChangeDim}
W^{k+1} + \widetilde\Lambda_{k+1} + \big(0^{(k-1)n},0^n,w_{i_j}\big) = W^{k+1} + \widetilde\Lambda_{k+1} + \big(0^{(k-1)n},-v_j,w_{i_j} + Bv_j\big)
\end{equation}
(we add a vector belonging to $\widetilde\Lambda_{k+1}$). 

We now apply the hypothesis of the lemma to the vectors $-v_1,\cdots,-v_{\ell}$: there exists $j\neq j'\in\llbracket 1,\ell\rrbracket$ such that
\begin{equation}\label{EqInter}
D_c\bigg(\Big(W^k + \widetilde\Lambda_k + (0^{(k-1)n} , -v_j)\Big) \cap \Big(W^k + \widetilde\Lambda_k + (0^{(k-1)n} , -v_{j'})\Big)\bigg) \ge D_k^\ell.
\end{equation}
Let $y$ be a point belonging to this intersection. Applying Equations~\eqref{EqIxIn} and \eqref{EqInter}, we get	that
\begin{equation}\label{EqLemDeFinCentr}
(y,x) \in W^{k+1} + \big(\widetilde\Lambda_k,0^n\big) + \big(0^{(k-1)n}, -v_j, w_{i_j} + Bv_j\big)
\end{equation}
and the same for $j'$ (this means that the point $(y,x)$ belongs to two different translates of the cube, see Figures \ref{FigLemDeFin} and \ref{FigLemDeFin2}).

Two different cases can occur.
\begin{enumerate}[(i)]
\item Either $i_j = i_{j'}$ (that is, the translation vectors $w_{i_j}$ and $w_{i_{j'}}$ are equal). As a consequence, applying Equation~\eqref{EqLemDeFinCentr}, we have
\begin{align*}
(y,x) + \big(0^{(k-1)n}, v_j, -Bv_j-w_{i_j}\big) \in & \Big( W^{k+1} + \big(\widetilde\Lambda_k,0^n\big) \Big)\cap\\
                                                     & \Big( W^{k+1} + \big(\widetilde\Lambda_k,0^n\big) + v'\Big),
\end{align*}
with
\[v' = \big(0^{(k-1)n}, -(v_{j'}-v_j), B(v_{j'}-v_j)\big) \in \widetilde\Lambda_{k+1} \setminus\widetilde\Lambda_k.\]
This implies that the set $W^{k+1} + \widetilde\Lambda_{k+1}$ auto-intersects (see Figure~\ref{FigLemDeFin}).
\item Or $i_j \neq i_{j'}$ (that is, $w_{i_j}\neq w_{i_{j'}}$). Combining Equations~\eqref{EqLemDeFinCentr} and \eqref{EqChangeDim} (note that $\big(\widetilde\Lambda_k,0^n\big) \subset \widetilde\Lambda_{k+1}$), we get
\[(y,x)\in \Big(W^{k+1} + \widetilde\Lambda_{k+1} + \big(0^{kn}, w_{i_j}\big)\Big) \cap \Big(W^{k+1} + \widetilde\Lambda_{k+1} + \big(0^{kn},w_{i_{j'}}\big)\Big).\]
This implies that two distinct vertical translates of $W^{k+1} + \widetilde\Lambda_{k+1}$ intersect (see Figure~\ref{FigLemDeFin2}).
\end{enumerate}

We have built points in some intersections, we now estimate the size of these intersections. Again, we have two cases.
\begin{enumerate}[(1)]
\item Either for more than the half of $\overline x \in U$ (for Lebesgue measure), we are in the case (i). To each of such $\overline x$ corresponds a translation vector $w_i$. We choose $w_i$ such that the set of corresponding $\overline x$ has the biggest measure; this measure is bigger than $\varep/\big(2(\ell-1)\big)\ge \varep/(2\ell)$. Reasoning as in the proof of Proposition~\ref{PerLin1}, and in particular using the notations of Corollary~\ref{CoroSansNom}, we get that the density $\widetilde C_1$ of the auto-intersection of $W^{k+1} + \widetilde\Lambda_{k+1} + (0,w_i)$ is bigger than $D_k^\ell\varep/(2\ell)$. This leads to (using  Corollary~\ref{CoroSansNom})
\[D_c(W^{k+1} + \widetilde\Lambda_{k+1}) \le D_c(W^{k} + \widetilde\Lambda_{k}) - \frac{D_k^\ell \varep}{4\ell}.\]
In this case, we get the conclusion (1) of the lemma.
\item  Or for more than the half of $\overline x \in U$, we are in the case (ii). Choosing the couple $(w_i,w_{i'})$ such that the measure of the set of corresponding $\overline x$ is the greatest, we get 
\[D_c\bigg(\Big(W^{k+1} + \widetilde\Lambda_{k+1} + (0^{kn} , w_i) \Big) \cap \Big(W^{k+1} + \widetilde\Lambda_{k+1} + (0^{kn} , w_{i'}) \Big)\bigg)\ge \frac{D_k^\ell \varep}{(\ell-1)(\ell-2)}.\]
In this case, we get the conclusion (2) of the lemma.
\end{enumerate}
\end{proof}

\begin{figure}[t]
\begin{center}
\includegraphics[width=.8\linewidth]{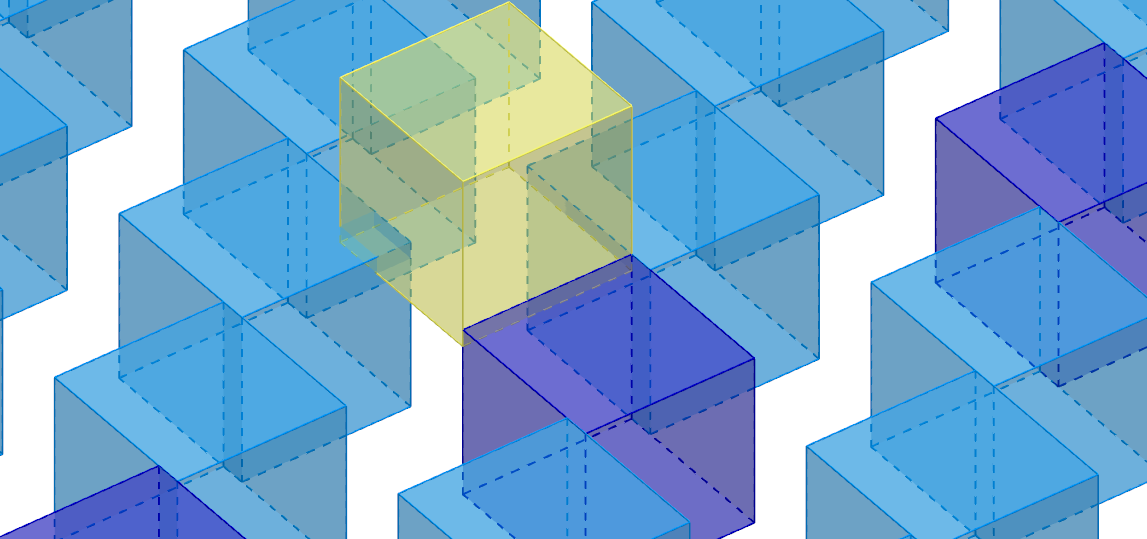}
\caption[Intersection of cubes, rate bigger than $1/3$]{Intersection of cubes in the case where the rate is bigger than $1/3$. The thickening of the cubes of $W^k + \widetilde\Lambda_k$ is represented in dark blue and the thickening of the rest of the cubes of $W^{k+1} + \widetilde\Lambda_{k+1}$ is represented in light blue; we have also represented another cube of $W^{k+2} + \widetilde\Lambda_{k+2}$ in yellow. We see that if the projection on the $z$-axis of the centre of the yellow cube is smaller than 1, then there is automatically an intersection between this cube and one of the blue cubes.}\label{FigInterCubes3DSupDemi}
\end{center}
\end{figure}

We can now prove Theorem~\ref{ConjPrincip}.

\begin{proof}[Proof of Theorem~\ref{ConjPrincip}]
As in the proof of Proposition~\ref{PerLin1}, we proceed by induction on $k$. Suppose that $\widetilde\Lambda_k$ is such that $D_c(W^k + \widetilde\Lambda_k)>1/\ell$. Then, Lemma~\ref{EstimTauxN} ensures that it is not possible to have $\ell$ disjoint translates of $W^k + \widetilde\Lambda_k$. Applying Lemma~\ref{LemDeFin}, we obtain that either $D_c(W^{k+1} + \widetilde\Lambda_{k+1})<D_c(W^k + \widetilde\Lambda_k)$, or 
it is not possible to have $\ell-1$ disjoint translates of $W^{k+1} + \widetilde\Lambda_{k+1}$. And so on, applying Lemma~\ref{LemDeFin} at most $\ell-1$ times, there exists $k'\in \llbracket k+1,k+\ell-1\rrbracket $ such that $W^{k'} + \widetilde\Lambda_{k'}$ has additional auto-intersections. Quantitatively, if we denote $\overline\tau^k = \overline\tau^k(B_1,\cdots,B_k)$, by Lemma \ref{LemDeFin}, either
\[\overline\tau^{k+1} \le \overline\tau^k - \frac{D_k^\ell \varep}{4\ell},\]
or
\[D_{k+1}^{\ell-1} \ge \frac{\varep}{(\ell-1)(\ell-2)}D_k^\ell \ge \frac{\varep}{\ell^2}D_k^\ell.\]

This leads to 
\[\overline\tau^{k+\ell-1} \le \overline\tau^k - \frac{\varep}{4\ell} \left(\frac{\varep}{\ell^2}\right)^{\ell-2} D_k^\ell.\]
But Lemma~\ref{EstimTauxN} implies that
\[D_k^\ell \ge \frac{\ell\overline\tau^k - 1}{\ell^2},\]
so 
\[\overline\tau^{k+\ell-1} \le \overline\tau^k - \frac{1}{4} \left(\frac{\varep}{\ell^2}\right)^{\ell-1} \left(\overline\tau^k - \frac{1}{\ell}\right),\]
thus, setting $\alpha_\ell =  1 - \frac14 (\frac{\varep}{\ell^2})^{\ell-1}$, we get 
\[\overline\tau^{k+\ell-1} - \frac{1}{\ell} \le  \alpha_\ell\left( \overline\tau^k - \frac{1}{\ell} \right).\]
This implies that for every $\ell>0$, the sequence of rates $\overline\tau^k$ is smaller than a sequence converging exponentially fast to $1/\ell$: we get that
\[\overline\tau^{\ell k}(A_1,\cdots,A_{\ell k}) \le \alpha_\ell^k + \frac{1}{\ell}.\]
In particular, the asymptotic rate of injectivity is generically equal to zero.
\bigskip

Finally, we quickly explain how to adapt these arguments to get the final part of the theorem. As $d\big(A_k,\bigcup_q V_q\big)>\delta$ for a positive asymptotic density $d$ of integers $k$, for any $k_0\in\N$, there exists an integer $m\in\N$ such that
\[\frac{1}{k_0+1}\card \left\{i\in\{m,\cdots,m+k_0\} \mid d\big(A_i,\bigcup_q V_q\big)>\delta \right\} \ge \frac{d}{2}.\]
We get the conclusion by following the first part of the proof of the theorem, and using additionally Lemma \ref{LemDeFin2} for the integers $i$ such that $d\big(A_i,\bigcup_q V_q\big)\le\delta$.
\end{proof}

\section[A local-global formula]{A local-global formula for generic expanding maps and generic diffeomorphisms}\label{SecRateExpand}

The goal of this section is to study the rate of injectivity of generic $C^r$-expanding maps and $C^r$-generic diffeomorphisms of the torus $\T^n$. Here, the term expanding map is taken from the point of view of discretizations: we say that a linear map $A$ is expanding if there exist no distinct integer points $x,y\in\Z^n$ such that $\widehat A(x) = \widehat A(y)$. This condition is satisfied in particular if for every $x\neq 0$, we have $\|Ax\|_\infty \le \|x\|_\infty$; thus when $n=1$ this definition coincides with the classical definition of expanding map.

The main result of this section is that the rate of injectivity of both generic $C^1$-diffeomorphisms and generic $C^r$-expanding maps of the torus $\T^n$ is obtained from a local-global formula (Theorems~\ref{convBisMieux} and \ref{TauxExpand}). Let us begin by explaining the case of diffeomorphisms.

\begin{theoreme}\label{convBisMieux}
Let $r\in[1,+\infty]$, and $f\in \Diff^r(\T^n)$ (or $f\in \Diff^r(\T^n,\Leb)$) be a generic diffeomorphism. Then $\tau^k(f)$ is well defined (that is, the limit superior in \eqref{EqTauT} is a limit) and satisfies:
\begin{align*}
\tau^k(f) & = \int_{\T^n} \tau^k\left(Df_x, \cdots, D f_{f^{k-1}(x)}\right) \ud \Leb(x)\\
   & = \int_{\T^n} \overline\tau^k\left(Df_x, \cdots, D f_{f^{k-1}(x)}\right) \ud \Leb(x)
\end{align*}
(with $\overline\tau$ defined by Equation \eqref{DeftauBar} page \pageref{DeftauBar}). Moreover, the function $\tau^k$ is continuous at $f$.
\end{theoreme}

The idea of the proof of this theorem is very simple: locally, the diffeomorphism is almost equal to a linear map. This introduces an intermediate \emph{mesoscopic} scale on the torus:
\begin{itemize}
\item at the macroscopic scale, the discretization of $f$ acts as $f$;
\item at the intermediate mesoscopic scale, the discretization of $f$ acts as a linear map;
\item at the microscopic scale, we are able to see that the discretization is a finite map and see the discreteness of the phase space.
\end{itemize}
This remark is formalized by Taylor's formula: for every $\varep>0$ and every $x\in\T^n$, there exists $\rho>0$ such that $f$ and its Taylor expansion at order 1 are $\varep$-close on $B(x,\rho)$. We then suppose that the derivative $Df_x$ is ``good'': the rate of injectivity of any of its $C^1$-small perturbations can be seen on a ball $B_R$ of $\R^n$ (with $R$ uniform in $x$). Then, the proof of the local-global formula is made in two steps.
\begin{itemize}
\item Prove that ``a lot'' of maps of $SL_n(\R)$ are ``good''. This is formalized by Lemma~\ref{LemTauxExpand}, which gives estimations of the size of the perturbations of the linear map allowed, and of the size of the ball $B_R$. Its proof is quite technical and uses crucially the formalism of model sets, and an improvement of Weyl's criterion.
\item Prove that for a generic diffeomorphism, the derivative satisfies the conditions of Lemma~\ref{LemTauxExpand} at almost every point. This follows easily from Thom's transversality theorem.
\end{itemize}

As the case of expanding maps is more involved but similar, we will prove the local-global formula only for expanding maps; the adaptation of it for diffeomorphisms is straightforward.
\bigskip

Remark that the hypothesis of genericity is necessary to get Theorem~\ref{convBisMieux}. For example, it can be seen that if we set 
\[f_0 = \left(\begin{matrix} \frac12 & -1 \\ \frac12 & 1 \end{matrix}\right),\]
then $\tau(f_0)=1/2$ whereas $\tau(f_0+(1/4,3/4)) = 3/4$. Thus, if $g$ is a diffeomorphism of the torus which is equal to $f_0+v$ on an open subset of $\T^2$, with $v$ a suitable translation vector, then the conclusions of Theorem~\ref{convBisMieux} do not hold (see Example 11.4 of \cite{Guih-These} for more explanations).
\bigskip

The definition of the linear analogue of the rate of injectivity of an expanding map in time $k$ is more complicated than for diffeomorphisms: in this case, the set of preimages has a structure of $d$-ary tree. We define the rate of injectivity of a tree --- with edges decorated by linear expanding maps --- as the probability of percolation of a random graph associated to this decorated tree (see Definition~\ref{Noel!}). In particular, if all the expanding maps are equal, then the connected component of the root of this random graph is a Galton-Watson tree. We begin by the definition of the set of expanding maps.

\begin{definition}\label{DefExpan}
For $r\ge 1$ and $d\ge 2$, we denote by $\mathcal D^r(\T^n)$\index{$\mathcal D^{r}(\T^n)$} the set of $C^r$ ``$\Z^n$-expanding maps'' of $\T^n$ for the infinite norm. More precisely, $\mathcal D^r(\T^n)$ is the set of maps $f : \T^n\to \T^n$, which are local diffeomorphisms, such that the derivative $f^{(\lfloor r\rfloor)}$ is well defined and belongs to $C^{r-\lfloor r\rfloor}(\T^n)$ and such that for every $x\in \T^n$ and every $v\in\Z^n\setminus\{0\}$, we have $\|Df_x v\|_\infty \ge 1$.

In particular, for $f\in \mathcal D^r(\T^n)$, the number of preimages of any point of $\T^n$ is equal to a constant, that we denote by $d$.
\end{definition}

Remark that in dimension $n=1$, the set $\mathcal D^r(\Sp^1)$ coincides with the classical set of expanding maps: $f\in\mathcal D^r(\Sp^1)$ if and only if it belongs to $C^r(\Sp^1)$ and $f'(x)\ge 1$ for every $x\in\Sp^1$.
\bigskip

We now define the linear setting corresponding to a map $f\in\mathcal D(\T^n)$. 

\begin{definition}
We set (see also Figure~\ref{TreeBee})\index{$I_k$}
\[I_k = \bigsqcup_{m=1}^k \llbracket 1,d\rrbracket^m\]
the set of $m$-tuples of integers of $\llbracket 1,d\rrbracket$, for $1\le m\le k$.

For $\ind = (i_1,\cdots,i_m)\in \llbracket 1,d\rrbracket^m$, we set $\len(\ind) = m$\index{$\len(\ind)$} its length and $\fat(\ind) = (i_1,\cdots,i_{m-1})\in \llbracket 1,d\rrbracket^{m-1}$\index{$\fat(\ind)$} its father (with the convention $\fat(i_1) = \emptyset$). The tuple $\fat^m(\ind)$ is defined inductively by the formula $\fat^{m+1}(\ind) = \fat(\fat^m(\ind))$.
\end{definition}

\begin{figure}[!b]
\begin{minipage}[c]{.35\linewidth}
\begin{center}
\begin{tikzpicture}[scale=.7]
\node (O) at (1,0){$\emptyset$};
\node (A) at (3,1){$(1)$};
\node (B) at (3,-1){$(2)$};
\node (C) at (6,1.5){$(1,1)$};
\node (D) at (6,.5){$(1,2)$};
\node (E) at (6,-.5){$(2,1)$};
\node (F) at (6,-1.5){$(2,2)$};
\draw (O) -- (A);
\draw (A) -- (C);
\draw (A) -- (D);
\draw (O) -- (B);
\draw (B) -- (E);
\draw (B) -- (F);
\end{tikzpicture}
\caption[The tree $T_2$ for $d=2$]{The tree $T_2$ for $d=2$.}\label{TreeBee}
\end{center}
\end{minipage}\hfill
\begin{minipage}[c]{.6\linewidth}
\begin{center}
\begin{tikzpicture}[scale=1.2]
\node (O) at (0,0){$y$};
\node (A) at (3,1){$x_{(1)}$};
\node (B) at (3,-1){$x_{(2)}$};
\node (C) at (6,1.5){$x_{(1,1)}$};
\node (D) at (6,.5){$x_{(1,2)}$};
\node (E) at (6,-.5){$x_{(2,1)}$};
\node (F) at (6,-1.5){$x_{(2,2)}$};
\draw (O) -- (A) node[sloped, midway, above]{$\det Df_{x_{(1)}}^{-1}$};
\draw (A) -- (C) node[sloped, midway, above]{$\det Df_{x_{(1,1)}}^{-1}$};
\draw (A) -- (D) node[sloped, midway, below]{$\det Df_{x_{(1,2)}}^{-1}$};
\draw (O) -- (B) node[sloped, midway, below]{$\det Df_{x_{(2)}}^{-1}$};
\draw (B) -- (E) node[sloped, midway, above]{$\det Df_{x_{(2,1)}}^{-1}$};
\draw (B) -- (F) node[sloped, midway, below]{$\det Df_{x_{(2,2)}}^{-1}$};
\end{tikzpicture}
\caption[Probability tree associated to the preimages of $y$]{The probability tree associated to the preimages of $y$, for $k=2$ and $d=2$. We have $f(x_{(1,1)}) = f(x_{(1,2)}) = x_{(1)}$, etc.}\label{ProbTree}
\end{center}
\end{minipage}
\end{figure}

The set $I_k$ is the linear counterpart of the set $\bigsqcup_{m=1}^k f^{-m}(y)$. Its cardinal is equal to $d(1-d^k)/(1-d)$.

\begin{definition}\label{Noel!}
Let $k\in\N$. The \emph{complete tree of order $k$} is the rooted $d$-ary tree $T_k$\index{$T_k$} whose vertices are the elements of $I_k$ together with the root, denoted by $\emptyset$, and whose edges are of the form $(\fat(\ind),\ind)_{\ind\in I_k}$ (see Figure~\ref{TreeBee}).

Let $(p_\ind)_{\ind\in I_k}$ be a family of numbers belonging to $[0,1]$. These probabilities will be seen as decorations of the edges of the tree $T_k$. We will call \emph{random graph associated to $(p_\ind)_{\ind\in I_k}$} the random subgraph $G_{(p_\ind)_\ind}$\index{$G_{(p_\ind)}$} of $T_k$, such that the laws of appearance the edges $(\fat(\ind),\ind)$ of $G_{(p_\ind)_\ind}$ are independent Bernoulli laws of parameter $p_\ind$. In other words, $G_{(p_\ind)_\ind}$ is obtained from $T_k$ by erasing independently each vertex of $T_k$ with probability $1-p_\ind$.

We define the \emph{mean density} $\overline D((p_\ind)_\ind)$\index{$\overline D((p_\ind)_\ind)$} of $(p_\ind)_{\ind\in I_k}$ as the probability that in $G_{(p_\ind)_\ind}$, there is at least one path linking the root to a leaf.
\end{definition}

Remark that if the probabilities $p_\ind$ are constant equal to $p$, the random graph $G_{(p_\ind)_\ind}$ is a Galton-Watson tree, where the probability for a vertex to have $i$ children is equal to $\binom{d}{i} p^i(1-p)^{d-i}$.

\begin{definition}\label{DefTreeMap}
By the notation $\overline D( (\det Df_x^{-1})_{x\in f^{-m}(y),\, 1\le m\le k})$, we will mean that the mean density is  taken with respect to the random graph $G_{f,y}$ associated to the decorated tree whose vertices are the $f^{-m}(y)$ for $0\le m\le k$, and  whose edges are of the form $(f(x),x)$ for $x\in f^{-m}(y)$ with $1\le m \le k$, each one being decorated by the number $\det Df_x^{-1}$ (see Figure~\ref{ProbTree}).
\end{definition}

Recall that the rates of injectivity are defined by (see also Definition~\ref{DefTauxDiffeo})
\[\tau^k(f_N) = \frac{\card\big((f_N)^k(E_N)\big)}{\card(E_N)} \qquad \text{and} \qquad \tau^k(f) = \limsup_{N\to+\infty} \tau^k(f_N).\]

\begin{theoreme}\label{TauxExpand}
Let $r\ge 1$, $f$ a generic element of $\mathcal D^r(\T^n)$ and $k\in\N$. Then, $\tau^k(f)$ is a limit (that is, the sequence $(\tau^k(f_N))_N$ converges), and we have
\begin{equation}\label{EqIntMieux}
\tau^k(f) = \int_{\T^n} \overline D\big( (\det Df_x^{-1})_{\begin{subarray}{l} 1\le m\le k \\ x\in f^{-m}(y)\end{subarray}}\big) \ud \Leb(y).
\end{equation}
Moreover, the map $f\mapsto \tau^k(f)$ is continuous at $f$.
\end{theoreme}

The proof of Theorem~\ref{TauxExpand} is mainly based on the following lemma, which treats the linear corresponding case. Its statement is divided into two parts, the second one being a quantitative version of the first.

\begin{lemme}\label{LemTauxExpand}
Let $k\in\N$, and a family $(A_\ind)_{\ind\in I_k}$ of invertible matrices, such that for any $\ind\in I_k$ and any $v\in\Z^n\setminus\{0\}$, we have $\|A_\ind v\|_\infty \ge 1$.

If the image of the map
\[\Z^n\ni x \mapsto \bigoplus_{\ind\in I_k} A_\ind^{-1} A_{\fat(\ind)}^{-1} \cdots A_{\fat^{\len(\ind)}(\ind)}^{-1} x\]
projects on a dense subset of the torus $\R^{n\card I_k}/\Z^{n\card I_k}$, then we have
\[D_d\left(\bigcup_{\ind\in \llbracket 1,d\rrbracket^k}\big( \widehat A_{\fat^{k-1}(\ind)} \circ \cdots \circ \widehat A_\ind \big) (\Z^n)\right) = \overline D\big((\det A_\ind^{-1})_\ind\big).\]

More precisely, for every $\ell', c\in\N$, there exists a locally finite union of positive codimension submanifolds $V_q$ of $(GL_n(\R))^{\card I_k}$ satisfying: for every $\eta'>0$, there exists a radius $R_0>0$ such that if $(A_\ind)_{\ind\in I_k}$ satisfies $d((A_\ind)_\ind, V_q)>\eta'$ for every $q$, then for every $R\ge R_0$, and every family $(v_\ind)_{\ind\in I_k}$ of vectors of $\R^n$, we have\footnote{The map $\pi(A+v)$\index{$\pi(A+v)$} is the discretization of the affine map $A+v$. The quantity $D_R$ is defined in Definition \ref{DefTaux}.}
\begin{equation}\label{EqLem22}
\left|D_R\left(\bigcup_{\ind\in \llbracket 1,d\rrbracket^k}\big( \pi(A_{\fat^{k-1}(\ind)} + v_{\fat^{k-1}(\ind)}) \circ \cdots \circ \pi(A_\ind + v_\ind) \big) (\Z^n) \right) - \overline D\big((\det A_\ind^{-1})_\ind\big)\right| < \frac{1}{\ell'}
\end{equation}
(the density of the image set is ``almost invariant'' under perturbations by translations), and for every $m\le k$ and every $\ind\in \llbracket 1,d\rrbracket^k$, we have\footnote{Where $(\Z^n)'$ stands for the set of points of $\R^n$ at least one coordinate of which belongs to $\Z+1/2$.}
\begin{equation}\label{EqLem222}
D_R\left\{x\in \big(A_{\fat^m(\ind)} + v_{\fat^m(\ind)}\big)(\Z^n)\ \middle\vert\ d\big(x,(\Z^n)' \big) < \frac{1}{{c}\ell'(2n+1)} \right\} < \frac{1}{{c}\ell'}
\end{equation}
(there is only a small proportion of the points of the image sets which are obtained by discretizing points close to $(\Z^n)'$).
\end{lemme}

Remark that this statement is purely linear: again, the difficulties of local-global formula's proof hold in the linear world.

The local-global formula~\eqref{EqIntMieux} will later follow from this lemma, an appropriate application of Taylor's theorem and Thom's transversality theorem (Lemma~\ref{PerturbCr}).

The next lemma uses the strategy of proof of Weyl's criterion to get a uniform convergence in Birkhoff's theorem for rotations of the torus $\T^n$ whose rotation vectors are outside of a neighbourhood of a finite union of hyperplanes.

\begin{lemme}[Weyl]\label{Weyl}
Let $\dist$ be a distance generating the weak-* topology on $\Prb$ the space of Borel probability measures on $\T^n$. Then, for every $\varep>0$, there exists a locally finite family of affine hyperplanes $H_i \subset \R^n$, such that for every $\eta>0$, there exists $M_0\in\N$, such that for every $\boldsymbol\lambda \in \R^n$ satisfying $d(\boldsymbol\lambda,H_q)>\eta$ for every $q$, and for every $M\ge M_0$, we have
\[\dist \left(\frac{1}{M}\sum_{m=0}^{M-1} \bar\delta_{m\boldsymbol\lambda }\,,\ \Leb_{\R^n/\Z^n}\right) < \varep,\]
where $\bar\delta_x$ is the Dirac measure of the projection of $x$ on $\R^n/\Z^n$.
\end{lemme}

\begin{proof}[Proof of Lemma~\ref{Weyl}]
As $\dist$ generates the weak-* topology on $\Prb$, it can be replaced by any other distance also generating the weak-* topology on $\Prb$. So we consider the distance $\dist_W$\index{$\dist_W$} defined by:
\[\dist_W(\mu,\nu) = \sum_{\boldsymbol k\in\N^n} \frac{1}{2^{k_1+\cdots+k_n}}\left| \int_{\R^n/\Z^n} e^{i2\pi\boldsymbol k\cdot x} \ud(\mu-\nu)(x)\right|;\]
there exists $K>0$ and $\varep'>0$ such that if a measure $\mu\in\Prb$ satisfies
\begin{equation}\label{EqWeyl}
\forall \boldsymbol k\in\N^n :\, 0<k_1+\cdots+k_n \le K, \quad \left| \int_{\R^n/\Z^n} e^{i2\pi\boldsymbol k\cdot x} \ud\mu(x)\right|<\varep',
\end{equation}
then $\dist(\mu,\Leb)<\varep$.

For every $\boldsymbol k\in\N^n\setminus\{0\}$ and $j\in\Z$, we set
\[H_{\boldsymbol k}^j = \{\boldsymbol\lambda\in\R^n \mid \boldsymbol k\cdot \boldsymbol\lambda = j\}.\]
Remark that the family $\{H_{\boldsymbol k}^j\}$, with $j\in\Z$ and $\boldsymbol k$ such that $0<k_1+\cdots+k_n \le K$, is locally finite. We denote by $\{H_q\}_q$ this family, and choose $\boldsymbol\lambda \in \R^n$ such that $d(\boldsymbol \lambda,H_q) > \eta$ for every $q$. We also take
\begin{equation}\label{DefMoe}
M_0 \ge \frac{2}{\varep' \left|1-e^{i2\pi\eta}\right|}.
\end{equation}
Thus, for every $\boldsymbol k\in\N^n$ such that $k_1+\cdots+k_n \le K$, and every $M\ge M_0$, the measure
\[\mu = \frac{1}{M}\sum_{m=0}^{M-1} \bar\delta_{m\boldsymbol\lambda}.\]
satisfies
\[\left| \int_{\R^n/\Z^n} e^{i2\pi\boldsymbol k\cdot x} \ud\mu(x)\right| = \frac{1}{M} \left|\frac{1-e^{i2\pi M\boldsymbol k\cdot \boldsymbol\lambda}}{1-e^{i2\pi \boldsymbol k\cdot \boldsymbol\lambda}}\right| \le \frac{2}{M_0} \frac{1}{\left|1-e^{i2\pi \boldsymbol k\cdot \boldsymbol\lambda}\right|}.\]
By \eqref{DefMoe} and the fact that $d(\boldsymbol k\cdot \boldsymbol\lambda,\Z)\ge \eta$, we deduce that
\[\left| \int_{\R^n/\Z^n} e^{i2\pi\boldsymbol k\cdot x} \ud\mu(x)\right| \le\varep'.\]
Thus, the measure $\mu$ satisfies the criterion \eqref{EqWeyl}, which proves the lemma.
\end{proof}

\begin{proof}[Proof of Lemma~\ref{LemTauxExpand}]
To begin with, let us treat the case $d=1$. Let $A_1,\cdots,A_k$ be $k$ invertible matrices. We want to compute the rate of injectivity of $\widehat A_k \circ \cdots \circ \widehat A_1$. Resuming the proof of Proposition~\ref{CalculTauxModel}, we see that $x\in \big( \widehat A_k \circ \cdots \circ \widehat A_1 \big) (\Z^n)$ if and only if $(0^{n(k-1)},x) \in W^k + \widetilde \Lambda_k$ (see page \pageref{DefMat1} for the definitions of these notations). This implies the following statement.

\begin{lemme}\label{ppe}
We have
\begin{equation}\label{densitete}
\det(A_k\cdots A_1)D_d\big( \widehat A_k \circ \cdots \circ \widehat A_1 \big) (\Z^n) = \nu(\operatorname{pr}_{\R^{nk}/\widetilde \Lambda_k}(W^k)),
\end{equation}
where $\nu$ is the uniform measure on the submodule $\operatorname{pr}_{\R^{nk}/\widetilde \Lambda_k}(0^{n(k-1)},\Z^n)$ of $\R^{nk}/\widetilde \Lambda_k$.
\end{lemme}

In particular, if the image of the map
\[\Z^n\ni x \mapsto \bigoplus_{m=1}^k (A_m)^{-1}\cdots(A_k)^{-1} x\]
projects on a dense subset of the torus $\R^{nk}/\Z^{nk}$, then the quantity~\eqref{densitete} is equal to the volume the projection of $W^k$ on $\R^{nk}/\widetilde\Lambda_k$ (see the end of the proof of Proposition~\ref{CalculTauxModel} and in particular the form of the matrix $\widetilde M_{A_1,\cdots,A_k}^{-1}$). By the hypothesis made on the matrices $A_m$ --- that is, for any $v\in\Z^n\setminus\{0\}$, $\|A_m v\|_\infty \ge 1$ --- this volume is equal to $1$ (simply because the restriction to $W^k$ of the projection $\R^{nk}\mapsto \R^{nk}/\widetilde\Lambda_k$ is injective). Thus, the density of the set $\big( \widehat A_k \circ \cdots \circ \widehat A_1 \big) (\Z^n)$ is equal to $1/(\det(A_k \cdots A_1))$  .

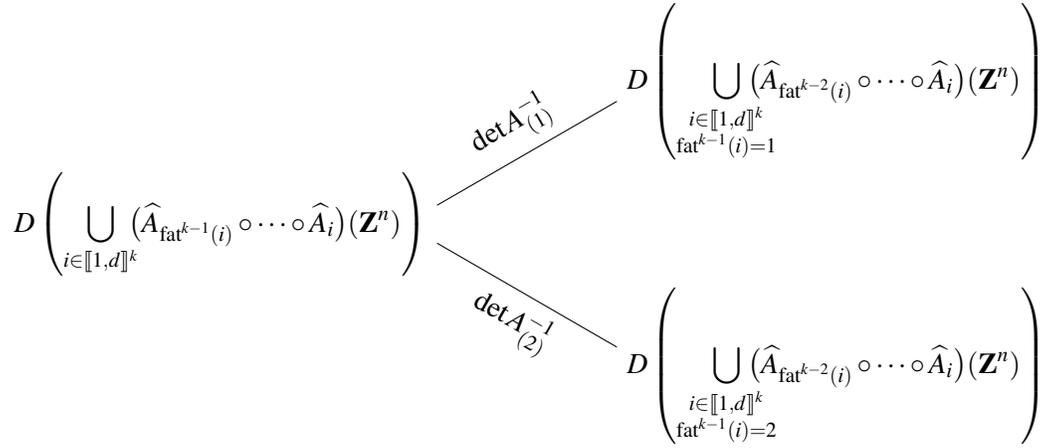
\begin{figure}[t]
\begin{center}
\begin{tikzpicture}[scale=1.26]
\node (O) at (0,0){$\displaystyle D\left(\bigcup_{\ind\in \llbracket 1,d\rrbracket^k}\!\!\!\! \big( \widehat A_{\fat^{k-1}(\ind)} \circ \cdots \circ \widehat A_\ind \big) (\Z^n)\right)$};
\node (A) at (6.5,1.5){$\displaystyle D\left(\bigcup_{\substack{\ind\in \llbracket 1,d\rrbracket^k \\ \fat^{k-1}(\ind) = 1}}\!\!\!\!\!\!\!\!\!\! \big( \widehat A_{\fat^{k-2}(\ind)} \circ \cdots \circ \widehat A_\ind \big) (\Z^n)\right)$};
\node (B) at (6.5,-1.5){$\displaystyle D\left(\bigcup_{\substack{\ind\in \llbracket 1,d\rrbracket^k \\ \fat^{k-1}(\ind) = 2}}\!\!\!\!\!\!\!\!\!\! \big( \widehat A_{\fat^{k-2}(\ind)} \circ \cdots \circ \widehat A_\ind \big) (\Z^n)\right)$};
\draw (O.5) -- (A.185) node[sloped, midway, above]{$\det A_{(1)}^{-1}$};
\draw (O.-5) -- (B.175) node[sloped, midway, below]{$\det A_{(2)}^{-1}$};
\end{tikzpicture}
\end{center}
\caption[Calculus of the density at the level $k$]{Calculus of the density of the image set at the level $k$ according to the density of its sons.}\label{ProbTree2}
\end{figure}
\bigskip

We now consider the general case where $d$ is arbitrary. We take a family $(A_\ind)_{\ind\in I_k}$ of invertible matrices, such that for any $\ind\in I_k$ and any $v\in\Z^n\setminus\{0\}$, we have $\|A_\ind v\|_\infty \ge 1$. A point $x\in\Z^n$ belongs to
\begin{equation}\label{GROS}
\bigcup_{\ind\in \llbracket 1,d\rrbracket^k}\big( \widehat A_{\fat^{k-1}(\ind)} \circ \cdots \circ \widehat A_\ind \big) (\Z^n)
\end{equation}
if and only if there exists $\ind\in \llbracket 1,d\rrbracket^k$ such that $(0^{m-1},x) \in W^k + \widetilde \Lambda_\ind$. Equivalently, a point $x\in\Z^n$ does not belong to the set \eqref{GROS} if and only if for every $\ind\in \llbracket 1,d\rrbracket^k$, we have $(0^{n(k-1)},x) \notin W^k + \widetilde \Lambda_\ind$. Thus, if the image of the map
\[\Z^n\ni x \mapsto \bigoplus_{\ind\in I_k} A_\ind^{-1} A_{\fat(\ind)}^{-1} \cdots A_{\fat^{\len(\ind)}(\ind)}^{-1} x\]
projects on a dense subset of the torus $\R^{n\card I_k}/\Z^{n\card I_k}$, then the events $x\in S_i$, with
\[S_i = \bigcup_{\substack{\ind\in \llbracket 1,d\rrbracket^k \\ \fat^{k-1}(\ind) = i}}\big( \widehat A_{\fat^{k-1}(\ind)} \circ \cdots \circ \widehat A_\ind \big) (\Z^n)\]
are independent (see Figure~\ref{ProbTree2}), meaning that for every $F\subset \llbracket 1,d\rrbracket$, we have
\begin{equation}\label{independant}
D-d\left(\bigcap_{i\in F} S_i\right) = \prod_{i\in F} D_d\big(S_i \big).
\end{equation}
Thus, by the inclusion-exclusion principle, we deduce that
\[D_d\left(\bigcup_{i\in \llbracket 1,d\rrbracket} S_i\right) = \sum_{\emptyset \neq F \subset \llbracket 1,d\rrbracket} (-1)^{\card(F)+1} \prod_{i\in F} D_d\big(S_i \big).\]
Moreover, the fact that for any $\ind\in I_k$ and any $v\in\Z^n\setminus\{0\}$, we have $\|A_\ind v\|_\infty \ge 1$ leads to
\[D_d(S_i) = \det A_{\fat^{k-1}(\ind)}^{-1}\ D_d\left(\bigcup_{\substack{\ind\in \llbracket 1,d\rrbracket^k \\ \fat^{k-1}(\ind) = i}}\big( \widehat A_{\fat^{k-2}(\ind)} \circ \cdots \circ \widehat A_\ind \big) (\Z^n)\right).\]

These facts imply that the density we look for follows the same recurrence relation as $\overline D\big((\det A_\ind^{-1})_\ind\big)$, thus
\[D_d\left(\bigcup_{\ind\in \llbracket 1,d\rrbracket^k}\big( \widehat A_{\fat^{k-1}(\ind)} \circ \cdots \circ \widehat A_\ind \big) (\Z^n)\right) = \overline D\big((\det A_\ind^{-1})_\ind\big).\]
\bigskip

The second part of the lemma is an effective improvement of the first one. To obtain the bound~\eqref{EqLem22}, we combine Lemma~\ref{Weyl} with Lemma~\ref{ppe} to get that for every $\varep>0$, there exists a locally finite collection of submanifolds $V_q$ of $(GL_n(\R))^{\card I_k}$ with positive codimension, such that for every $\eta'>0$, there exists $R_0>0$ such that if $d((A_\ind)_\ind,V_q)>\eta'$ for every $q$, then Equation~\eqref{independant} is true up to $\varep$.

The other bound~\eqref{EqLem222} is obtained independently from the rest of the proof by a direct application of Lemma~\ref{Weyl} and of Lemma~\ref{ppe} applied to $k=1$.
\end{proof}

The following lemma is a direct application of Thom's transversality theorem.

\begin{lemme}[Perturbations in $C^r$ topology]\label{PerturbCr}
Let $1\le r \le +\infty$ and $f$ a generic element of $\mathcal D^r(\T^n)$. Then, for every $k\in\N$, every $\ell'\in\N$ and every finite collection $(V_q)$ of submanifolds of positive codimension of $(GL_n(\R))^{dm}$, there exists $\eta>0$ such that the set
\[T_\eta = \left\{ y\in\T^n\ \middle\vert\ \forall q,\, d\Big( \big( Df_{x}\big)_{\begin{subarray}{l} 1\le m \le k \\ x\in f^{-m}(y)\end{subarray}}\, ,\, V_q\Big) > \eta \right\}\]
contains a finite disjoint union of cubes\footnote{here, a cube is just any ball for the infinite norm.}, whose union has measure bigger than $1-1/\ell'$.
\end{lemme}

\begin{proof}[Proof of Lemma~\ref{PerturbCr}]
By Thom's transversality theorem, for a generic map $f\in\mathcal D^r(\T^n)$, the set
\[\left\{ y\in\T^n\ \middle\vert\ \exists q,\, \big(Df_{x}\big)_{\begin{subarray}{l} 1\le m \le k \\ x\in f^{-m}(y)\end{subarray}}\in V_q \right\}\]
is a locally finite union of positive codimension submanifolds. Thus, the sets $T_\eta^{\scriptscriptstyle\complement}$ are compact sets and their (decreasing) intersection over $\eta$ is a locally finite union of positive codimension submanifolds. So, there exists $\eta>0$ such that $T_\eta^{\scriptscriptstyle\complement}$ is close enough to this set for Hausdorff topology to have the conclusions of the lemma.
\end{proof}

We can now begin the proof of Theorem~\ref{TauxExpand}.

\begin{proof}[Proof of Theorem~\ref{TauxExpand}]
Let $f$ be a generic element of $\mathcal D^r(\T^n)$. The idea is to cut the torus $\T^n$ into small pieces on which $f$ is very close to its order $1$ Taylor expansion.

Let $m\in\N$ and $\ell\in\N^*$. We want to prove that the set of accumulation points of the sequence $(\tau^m_N(f))_N$ is included in the ball of radius $1/\ell$ and centre
\[\int_{\T^n} \overline D\big( (Df_x)_{\begin{subarray}{l} 1\le m\le k \\ x\in f^{-m}(y)\end{subarray}}\big) \ud \Leb(y).\]
(that is, the right side of Equation~\eqref{EqIntMieux}).

To do that, we first set $\ell' = 3\ell$ and $c=d(1-d^k)/(1-d) = \card(I_k)$, and use Lemma~\ref{LemTauxExpand} to get a locally finite union of positive codimension submanifolds $V_q$ of $(GL_n(\R))^{\card(I_k)}$. We then apply Lemma~\ref{PerturbCr} to these submanifolds and to $\ell' = 4\ell$; as $f$ is generic this gives us a parameter $\eta>0$ such that the set
\[\left\{ y\in\T^n\ \middle\vert\ \forall q,\, d\Big( \big( Df_{x}\big)_{\begin{subarray}{l} 1\le m\le k \\ x\in f^{-m}(y)\end{subarray}}\, ,\, V_q\Big)<\eta \right\}\]
is contained in a disjoint finite union $\mathcal C$ of cubes, whose union has measure smaller than $1/(4\ell)$. Finally, we apply Lemma~\ref{LemTauxExpand} to $\eta' = \eta/2$; this gives us a radius $R_0>0$ such that if $(A_\ind)_{\ind \in I_k}$ is a family of matrices of $GL_n(\R)$ satisfying $d((A_\ind)_\ind, V_q)>\eta/2$ for every $q$, then for every $R\ge R_0$, and every family $(v_\ind)_{\ind\in I_k}$ of vectors of $\R^n$, we have
\begin{equation}\label{DiffDensGene}
\left|D_R\left(\bigcup_{\ind\in \llbracket 1,d\rrbracket^k}\big( \pi(A_{\fat^{k-1}(\ind)} + v_{\fat^{k-1}(\ind)}) \circ \cdots \circ \pi(A_\ind + v_\ind) \big) (\Z^n) \right) - \overline D\big((\det A_\ind^{-1})_\ind\big)\right| < \frac{1}{3\ell},
\end{equation}
and for every $i,j$,
\begin{equation}\label{ModuloGene}
D_R\left\{x\in \big(A_{j^m(\ind)} + v_{j^m(\ind)}\big)(\Z^n)\ \middle\vert\ d\big(x,(\Z^n)' \big) < \frac{1}{3\ell(2n+1)\card I_k} \right\} < \frac{1}{3\ell\card I_k}.
\end{equation}
\bigskip

Remark that if $\delta'$ is small enough, then the set
\[\left\{ y\in\T^n\ \middle\vert\ \forall q,\, d\Big( \big( Df_{x}\big)_{\begin{subarray}{l} 1\le m\le k \\ x\in f^{-m}(y)\end{subarray}}\, ,\, V_q\Big)>\eta/2 \right\}\]
contains a set $\mathcal C'$, which is a finite union of cubes whose union has measure bigger than $1-1/(3\ell)$.

Let $C$ be a cube of $\mathcal C'$, $y\in C$ and $x\in f^{-m}(y)$, with $1\le m \le k$. We write the Taylor expansion of order 1 of $f$ at the neighbourhood of $x$; by compactness we obtain
\[\sup\left\{\frac{1}{\|z\|}\big\|f(x+z)-f(x)-Df_x(z)\big\|\ \middle|\ x\in C,\,z\in B(0,\rho)\right\}\underset{\rho\to 0}{\longrightarrow}0.\]
Thus, for every $\varep > 0$, there exists $\rho>0$ such that for all $x\in C$ and all $z\in B(0,\rho)$, we have
\begin{equation}\label{eqetata}
\left\|f(x+z)-\big(f(x)+Df_x(z)\big)\right\| < \varep\|z\| \le \varep \rho.
\end{equation}

We now take $R\ge R_0$. We want to find an order of discretization $N$ such that the error made by linearizing $f$ on $B(x,R/N)$ is small compared to $N$, that is, for every $z\in B(0,R/N)$, we have
\[\left\|f(x+z)-\big(f(x)+Df_x(z)\big)\right\| < \frac{1}{3\ell(2n+1)\card I_k}\cdot\frac{1}{N}.\]
To do that, we apply Bound~\eqref{eqetata} to
\[\varep = \frac{1}{3R\ell(2n+1)\card I_k},\]
to get a radius $\rho>0$ (we can take $\rho$ as small as we want), and we set $N = \lceil R/\rho \rceil$ (thus, we can take $N$ as big as we want). By~\eqref{eqetata}, for every $z\in B(0,R)$, we obtain the desired bound:
\[\left\|f(x+z/N)-\big(f(x)+Df_x(z/N)\big)\right\| < \frac{1}{3\ell(2n+1)\card I_k}\cdot\frac{1}{N}.\]
Combined with \eqref{ModuloGene}, this leads to
\begin{equation}\label{eqetatata}
\frac{\card\Big(f_N\big(B(x,R/N)\big)\, \Delta\, P_N\big(f(x) + Df_x (B(0,R/N))\big)\Big)}{\card\big(B(x,R/N) \cap E_N\big)} \le \frac{1}{3\ell \card I_k};
\end{equation}
in other words, on every ball of radius $R/N$, the image of $E_N$ by $f_N$ and the discretization of the linearization of $f$ are almost the same (that is, up to a proportion $1/(3\ell \card I_k)$ of points).

We now set $R_1 = R_0 \|f'\|_\infty^m$, and choose $R\ge R_1$, to which is associated a number $\rho>0$ and an order $N = \lceil R/\rho \rceil$, that we can choose large enough so that $2R/N \le \|f'\|_\infty$. We also choose $y\in C$. As
\[\card\big(f_N^m(E_N) \cap B(y,R/N) \big) = \card\left( \bigcup_{x\in f^{-m}(y)} f_N^m\big(B(x,R/N) \cap E_N\big) \cap B(y,R/N)\right),\]
and using the estimations~\eqref{DiffDensGene} and~\eqref{eqetatata}, we get
\[\left|\frac{\card\big(f_N^m(E_N) \cap B(y,R/N) \big)}{\card\big(B(y,R/N) \cap E_N\big)} -\overline D\big( (\det Df_x^{-1})_{\begin{subarray}{l} 1\le m\le k \\ x\in f^{-m}(y)\end{subarray}}\big)\right| < \frac{2}{3\ell}.\]
As such an estimation holds on a subset of $\T^n$ of measure bigger than $1-1/(3\ell)$, we get the conclusion of the theorem.
\end{proof}
\bigskip

We can easily adapt the proof of Lemma~\ref{LemTauxExpand} to the case of sequences of matrices, without expansiveness hypothesis.

\begin{lemme}\label{LemTauxDiff2}
For every $k\in\N$ and every $\ell', c\in\N$, there exists a locally finite union of positive codimension submanifolds $V_q$ of $(GL_n(\R))^{k}$ (respectively $(SL_n(\R))^{k}$) such that for every $\eta'>0$, there exists a radius $R_0>0$ such that if $(A_m)_{1\le m\le k}$ is a finite sequence of matrices of $(GL_n(\R))^{k}$ (respectively $(SL_n(\R))^{k}$) satisfying $d((A_m)_m, V_q)>\eta'$ for every $q$, then for every $R\ge R_0$, and every family $(v_m)_{1\le m \le k}$ of vectors of $\R^n$, we have
\[\left|D_R\left(\big( \widehat A_{k} \circ \cdots \circ \widehat A_1\big) (\Z^n) \right) - \det(A_k^{-1}\cdots A_1^{-1})\overline\tau^k(A_1,\cdots,A_k)\right| < \frac{1}{\ell'}\]
(the density of the image set is ``almost invariant'' under perturbations by translations), and for every $m\le k$, we have\footnote{Recall that $(\Z^n)'$ stands for the set of points of $\R^n$ at least one coordinate of which belongs to $\Z+1/2$.}
\[D_R\left\{x\in \big(A_{m} + v_{m}\big)(\Z^n)\ \middle\vert\ d\big(x,(\Z^n)' \big) < \frac{1}{{c}\ell'(2n+1)} \right\} < \frac{1}{{c}\ell'}\]
(there is only a small proportion of the points of the image sets which are obtained by discretizing points close to $(\Z^n)'$).
\end{lemme}

With the same proof as Theorem~\ref{TauxExpand}, Lemma~\ref{LemTauxDiff2} leads to the local-global formula for generic $C^r$-diffeo\-morphisms (Theorem~\ref{convBisMieux}).

\section{Asymptotic rate of injectivity for a generic conservative diffeo\-morphism}\label{ChapAsympto}

The goal of this section is to prove that for any $r\ge 1$, the degree of recurrence of a generic conservative $C^r$-diffeomorphism is equal to 0.

\begin{theoreme}\label{limiteEgalZero}
Let $r\in[1,+\infty]$. For a generic conservative diffeomorphism $f\in\Diff^r(\T^n,\Leb)$, we have
\[\lim_{k\to\infty}\tau^k(f)=0;\]
more precisely, for every $\varep>0$, the set of diffeomorphisms $f\in\Diff^r(\T^n,\Leb)$ such that $\lim_{t\to +\infty} \tau^k(f)<\varep$ is open and dense.

In particular\footnote{Using Equation~\eqref{intervLim} page~\pageref{intervLim}.}, we have $\lim_{N\to +\infty} \Dr(f_N) = 0$.
\end{theoreme}

It will be obtained by using the local-global formula (Theorem~\ref{convBisMieux}) and the result about the asymptotic rate of injectivity of a generic sequence of matrices (Theorem~\ref{ConjPrincip}). The application of the linear results will be made through Thom transversality arguments (Lemma \ref{PerturbCr}).

\begin{proof}[Proof of theorem \ref{limiteEgalZero}]
Let $\varep>0$, $r\in[0,+\infty]$ and consider a generic element $f$ of $\Diff^r(\T^n,\Leb)$. Applying the second part of Theorem \ref{ConjPrincip}, one gets a locally finite collection of submanifolds $V_q$ of $SL_n(\R)$, which allow to define the set
\[T_\eta = \left\{y\in\T^n \mid \forall q,\, d\big( Df_y,V_q\big) \ge\eta \right\}.\]
Then, Birkhoff theorem (non-ergodic version) applied to Lebesgue measure implies that the set $T_\eta^\infty$ of points $y\in\T^n$ such that
\[\limsup_{m\to +\infty}\frac{1}{m}\sum_{k=0}^{m-1} \1_{T_\eta}(f^k(x))\ge \frac12 \]
has Lebesgue measure greater than $2\Leb(T_\eta) - 1$.

Now, Lemma \ref{PerturbCr} (transversality) implies that as $f$ is a generic element of $\Diff^r(\T^n,\Leb)$, there exists $\eta>0$ such that $\Leb(T_\eta)>1-\varep$, in particular $\Leb(T_\eta^\infty)>1-2\varep$. Then, applying the second part of Theorem \ref{ConjPrincip}, we get that for all $y\in T_\eta^\infty$, we have $\overline\tau^\infty\big((Df_{f^k(y)})_k\big) \le\varep$. In particular, there exists $k_0\in\N$ and $T'\subset T_\eta^\infty$, with $\Leb(T') \ge 1-3\varep$ such that for all $y\in T'$, we have $\overline\tau^{k_0}\big((Df_{f^k(y)})_k\big) \le \varep$ (as the sets of points $y\in T_\eta^\infty$, such that  $\overline\tau^k\big((Df_{f^k(y)})_k\big) \le\varep$ have their increasing union over $k$ equal to $T_\eta^\infty$). Applying Theorem \ref{convBisMieux}, one gets that
\begin{align*}
\tau^k(f) & \le \int_{T'} \overline\tau^k\left(Df_x, \cdots, D f_{f^{k-1}(x)}\right) \ud \Leb(x) + \int_{{T'}^\complement} 1 \ud \Leb(x)\\
   & \le \varep + 3\varep = 4\varep.
\end{align*}
\end{proof}

\section{Asymptotic rate of injectivity for a generic dissipative diffeomorphism}\label{SecJairo}

Here we tackle the issue of the asymptotic rate of injectivity of generic dissipative diffeomorphisms. Again, we will consider the torus $\T^n$ equipped with Lebesgue measure $\Leb$ and the canonical measures $E_N$, see Section~\ref{AddendSett} for a more general setting where the result is still true. The study of the rate of injectivity for generic dissipative diffeomorphisms is based on the following theorem of A.~Avila and J.~Bochi.

\begin{theoreme}[Avila, Bochi]\label{ArturJairo}
Let $f$ be a generic $C^1$ maps of $\T^n$. Then for every $\varep>0$, there exists a compact set $K\subset \T^n$ and an integer $m\in\N$ such that
\[\Leb(K) > 1-\varep \qquad \text{and} \qquad \Leb(f^m(K))<\varep.\]
\end{theoreme}

This statement is obtained by combining Lemma 1 and Theorem 1 of \cite{MR2267725}.

\begin{rem}
As $C^1$ expanding maps of $\T^n$ and $C^1$ diffeomorphisms of $\T^n$ are open subsets of the set of $C^1$ maps of $\T^n$, the same theorem holds for generic $C^1$ expanding maps and $C^1$ diffeomorphisms of $\T^n$ (this had already been proved in the case of $C^1$-expanding maps by A.~Quas in \cite{MR1688216}).
\end{rem}

This theorem can be used to compute the asymptotic rate of injectivity of a generic diffeomorphism.

\begin{coro}\label{CoroCoroArturJairo}
The asymptotic rate of injectivity of a generic dissipative diffeomorphism $f\in\Diff^1(\T^n)$ is equal to 0. In particular, the degree of recurrence $\Dr(f_N)$ of a generic dissipative diffeomorphism tends to 0 when $N$ goes to infinity.
\end{coro}

The same statement holds for generic $C^1$ expanding maps.

\begin{proof}[Proof of Corollary~\ref{CoroCoroArturJairo}]
The proof of this corollary mainly consists in stating which good properties the compact set $K$ of Theorem~\ref{ArturJairo} can be supposed to possess. Thus, for $f$ a generic diffeomorphism and $\varep>0$, there exists $m>0$ and a compact set $K$ such that $\Leb(K) > 1-\varep$ and $\Leb(f^m(K))<\varep$.

First of all, it can be easily seen that Theorem~\ref{ArturJairo} is still true when the compact set $K$ is replaced by an open set $O$: simply consider an open set $O'\supset f^m(K)$ such that $\Leb(O')<\varep$ (by regularity of the measure) and set $O = f^{-m}(O') \supset K$. We then approach the set $O$ by unions of dyadic cubes of $\T^n$: we define the cubes of order $2^M$ on $\T^n$
\[C_{M,i} = \prod_{j=1}^n\left[\frac{i_j}{2^M},\frac{i_j+1}{2^M} \right],\]
and set\label{pageCubes}
\[U_M = \operatorname{Int}\left(\overline{\bigcup_{C_{M,i} \subset O} C_{M,i} }\right),\]
where $\operatorname{Int}$ denotes the interior. Then, the union $\bigcup_{M\in\N} U_M$ is increasing in $M$ and we have $\bigcup_{M\in\N} U_M = O$. In particular, there exists $M_0\in\N$ such that $\Leb(U_{M_0})>1-\varep$, and as $U_{M_0}\subset O$, we also have $\Leb(f^m(U_{M_0}))<\varep$. We denote $U = U_{M_0}$. Finally, as $U$ is a finite union of cubes, and as $f$ is a diffeomorphism, there exists $\delta>0$ such that the measure of the $\delta$-neighbourhood of $f^m(U)$ is smaller than $\varep$. We denote by $V$ this $\delta$-neighbourhood.

As $U$ is a finite union of cubes, there exists $N_0\in\N$ such that if $N\ge N_0$, then the proportion of points of $E_N$ which belong to $U$ is bigger than $1-2\varep$, and the proportion of points of $E_N$ which belong to $V$ is smaller than $2\varep$. Moreover, if $N_0$ is large enough, then for every $N\ge N_0$, and for every $x_N\in E_N\cap U$, we have $f_N^m(x_N)\in V$. This implies that
\[\frac{\card(f_N^m(E_N))}{\card(E_N)}\le 4\varep,\]
which proves the corollary.
\end{proof}

\section{Asymptotic rate of injectivity of a generic expanding map}

In this section, we prove that the asymptotic rate of injectivity of a generic expanding map is equal to 0. Note that a local version of this result was already obtained by P.P.~Flockermann in his thesis (Corollary~2 page 69 and Corollary~3 page 71 of \cite{Flocker}), stating that for a generic $C^{1+\alpha}$ expanding map $f$ of the circle, the ``local asymptotic rate of injectivity'' is equal to 0 almost everywhere. Some of his arguments will be used in this section. Note also that in $C^1$ regularity, the equality $\tau^\infty(f) = 0$ for a generic $f$ is a consequence of Theorem~\ref{ArturJairo} of A.~Avila and J.~Bochi (see also Corollary~\ref{CoroCoroArturJairo}); the same theorem even proves that the asymptotic rate of injectivity of a generic $C^1$ \emph{endomorphism} of the circle is equal to 0.

\begin{definition}
We define $Z_m$ as the number of children at the $m$-th generation in $G_{f,y}$ (see Definition~\ref{DefTreeMap}).
\end{definition}

\begin{prop}\label{PropTruc}
For every $r\in]1,+\infty]$, for every $f\in \mathcal D^r(\Sp^1)$ and every $y\in\Sp^1$, we have
\[\mathbf{P}(Z_m>0) \underset{m\to+\infty}{\longrightarrow}0.\]
Equivalently,
\[\overline D\big( (\det Df_x^{-1})_{\begin{subarray}{l} 1\le m\le k \\ x\in f^{-m}(y)\end{subarray}}\big) \underset{k\to+\infty}{\longrightarrow}0.\]
\end{prop}

\begin{lemme}\label{Espoir}
The expectation of $Z_m$ satisfies
\[\E(Z_m) = (\Ll^m 1)(y),\]
where $\Ll$ is the Ruelle-Perron-Frobenius associated to $f$ and $1$ denotes the constant function equal to 1 on $\Sp^1$.
In particular, there exists a constant $\Sigma_0>0$ such that $\E(Z_m) \le \Sigma_0$ for every $m\in\N$.
\end{lemme}

The second part of the lemma is deduced from the first one by applying the theorem stating that for every $C^r$ expanding map $f$ of $\Sp^1$ ($r>1$), the maps $\Ll^m 1$ converge uniformly towards a Hölder map, which is the density of the unique SRB measure of $f$ (see for example \cite{MR2504311}). The first assertion of the lemma follows from the convergence of the operators $f_N^*$ acting on $\Prb$ (the space of Borel probability measures) towards the Ruelle-Perron-Frobenius operator.

\begin{definition}
The \emph{transfer operator} associated to the map $f$ (usually called Ruelle-Perron-Fro\-be\-nius operator), which acts on densities of probability measures, will be denoted by $\Ll_f$\index{$\Ll_f$}. It is defined by
\[\Ll_f \phi(y) = \sum_{x\in f^{-1}(y)} \frac{\phi(x)}{f'(x)}.\]
\end{definition}

Lemma~\ref{Espoir} follows directly from the following lemma.

\begin{lemme}\label{LemRPF}
Denoting $\Leb_N$ the uniform measure on $E_N$, for every $C^1$ expanding map of $\Sp^1$ and every fixed $m\ge 0$, we have convergence of the measures $(f_N^*)^m(\Leb_N)$ towards the measure of density $\Ll_f^m 1$ (where $1$ denotes the constant function equal to 1).
\end{lemme}

The proof of this lemma is straightforward but quite long. We sketch here this proof, the reader will find a complete proof using generating functions in Section~3.4 of \cite{Flocker} and a quantitative version of it in Section~12.2 of \cite{Guih-These}.

\begin{proof}[Sketch of proof of Lemma~\ref{LemRPF}]
As $f$ is $C^1$, by the mean value theorem, for every segment $I$ small enough, we have
\[\left| \Leb(I) - \frac{\Leb(f(I))}{f'(x_0)}\right| \le \varep.\]
Moreover, for every interval $J$, 
\[\left|\Leb(J) - \frac{\card(J\cap E_N)}{\card(E_N)}\right| \le \frac{1}{N}.\]
These two inequalities allows to prove the local convergence of the measures $f_N^*(\lambda_N)$ towards the measure with density $\Ll_f 1$. The same kind of arguments holds in arbitrary times, and allows to prove the lemma.
\end{proof}

\begin{proof}[Proof of Proposition~\ref{PropTruc}]
We fix $\varep>0$, and set $K=\lceil\Sigma_0/\varep\rceil$ (the constant $\Sigma_0$ being given by Lemma~\ref{Espoir}) and
\[a_m = \mathbf{P}(Z_m = 0)\,,\quad  b_m = \mathbf{P}(0 < Z_m\le K)\,,\quad  c_m = \mathbf{P}(Z_m > K).\]

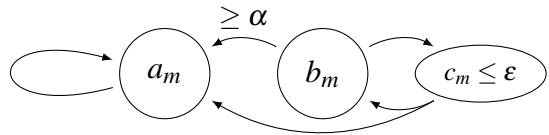
\begin{figure}[ht]
\begin{center}
\begin{tikzpicture}[scale=.7]
\node[draw,ellipse] (C) at (6,0) {\small$c_m\le \varep$};
\node[draw,circle,minimum height=1.2cm] (B) at (3,0) {\large $b_m$};
\node[draw,circle,minimum height=1.2cm] (A) at (0,0) {\large $a_m$};
\draw[->,>=latex,shorten >=3pt, shorten <=3pt] (C) to[bend left] (B);
\draw[->,>=latex,shorten >=3pt, shorten <=3pt] (B) to[bend left] (C);
\draw[->,>=latex,shorten >=3pt, shorten <=3pt] (C) to[bend left] (A);
\draw[->,>=latex,shorten >=3pt, shorten <=3pt] (B) to[bend right] node[midway, above]{$\ge \alpha$} (A) ;
\draw[->,>=latex,shorten >=3pt, shorten <=3pt] (A) to[loop left,looseness=16,min distance=10mm] (A);
\end{tikzpicture}
\end{center}
\caption{Transition graph for $Z_m$: $a_m = \mathbf{P}(Z_m = 0)$, $b_m = \mathbf{P}(0 < Z_m\le K)$, $c_m = \mathbf{P}(Z_m > K)$.}\label{GraphTransZ}
\end{figure}

\noindent We want to prove that the sequence $(a_m)_{m\in\N}$ tends to 1.

Of course, $a_m+b_m+c_m=1$ for any $m$. A generation with less than $K$ children will give birth to zero child with positive probability: we have
\[\mathbf{P}(Z_{m+1}=0 \mid 0 < Z_m\le K) \ge \left( 1 - \frac{1}{\|f'\|_\infty}\right)^{dK}.\]
In other words, setting $\alpha = ( 1 - \|f'\|_\infty^{-1})^{dK}$, we get $a_{m+1} \ge a_m + \alpha b_m$ (see also Figure~\ref{GraphTransZ}).

Furthermore, by Markov inequality and Lemma~\ref{Espoir}, we have
\[\mathbf{P}(Z_m\ge \Sigma_0/\varep) \le \varep,\]
so $c_m\le\varep$.

In summary, we have
\[\left\{\begin{array}{l}
a_m + b_m + c_m = 1\\
c_m\le \varep\\
a_{m+1} \ge a_m + \alpha b_m.
\end{array}\right.\]
This leads to $a_{m+1}\ge (1-\alpha) a_m + \alpha(1-\varep)$, which implies that $\liminf a_m \ge 1-\varep$. As this holds for any $\varep>0$, we get that $\lim a_m = 1$.
\end{proof}

\begin{coro}\label{Tau0Dilat}
For any $r\in ]1,+\infty]$, a generic map $f\in \mathcal D^r(\Sp^1)$ satisfies $\tau^\infty(f) = 0$. In particular, $\lim_{N\to +\infty} \Dr(f_N) = 0$.
\end{coro}

\begin{proof}[Proof of Corollary~\ref{Tau0Dilat}]
It is an easy consequence of the local-global formula (Theorem~\ref{TauxExpand}), Proposition~\ref{PropTruc} and the dominated convergence theorem.
\end{proof}

\section{Numerical simulations}

\begin{figure}[ht]
\begin{center}
\includegraphics[width=.5\linewidth,trim = .5cm .3cm .6cm .1cm,clip]{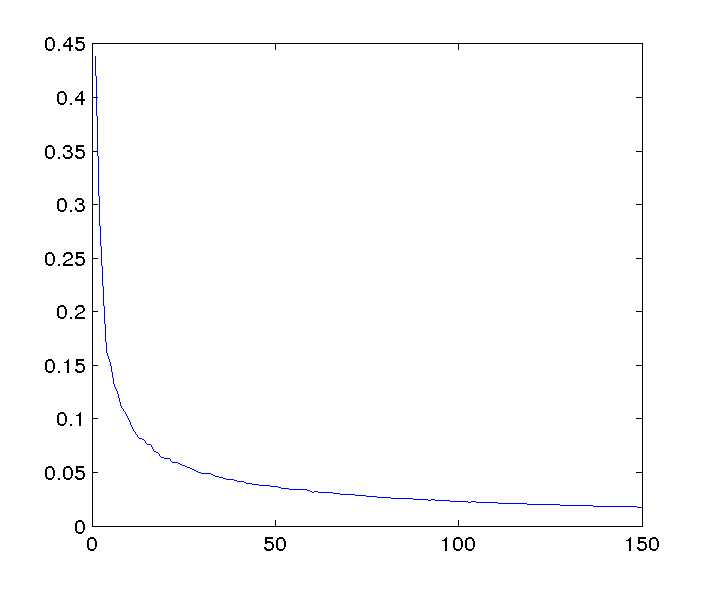}
\caption[Simulation of the degree of recurrence of $f_N$]{Simulation of the degree of recurrence $D(f_N)$ of the conservative diffeomorphism $f$, depending on $N$, on the grids $E_N$ with $N=128k$, $k=1,\cdots,150$.}\label{GrafDnDiffeoCons}
\end{center}
\end{figure}

\begin{figure}[ht]
\begin{center}
	\includegraphics[width=.5\linewidth,trim = .5cm .3cm .6cm .1cm,clip]{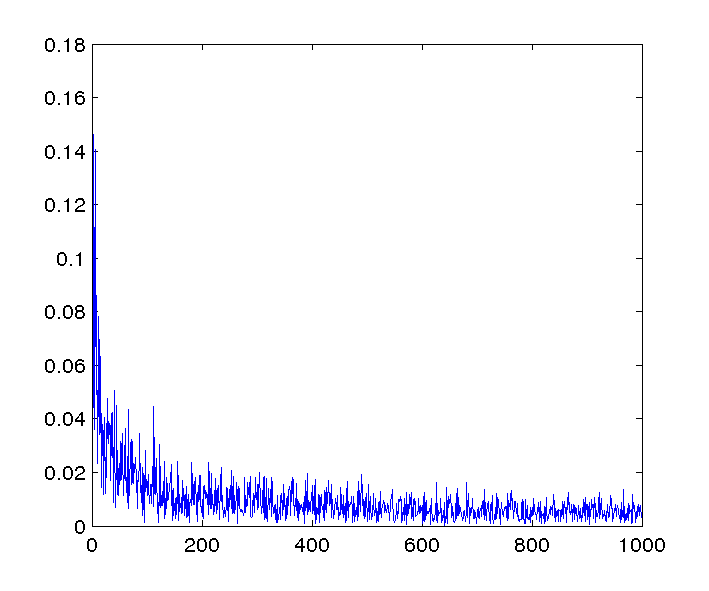}
\caption[Simulation of the degree of recurrence of $f_N$]{Simulation of the degree of recurrence $D(g_N)$ of the expanding map $g$, depending on $N$, on the grids $E_N$ with $N=128k$, $k=1,\cdots,1\,000$.}\label{GrafDnExp}
\end{center}
\end{figure}

We have computed numerically the degree of recurrence of a diffeomorphism $f$, which is $C^1$-close to $\Id$. It is defined by $f = Q\circ P$, with
\[P(x,y) = \big(x,y+p(x)\big)\quad\text{and}\quad Q(x,y) = \big(x+q(y),y\big),\]
\[p(x) = \frac{1}{209}\cos(2\pi\times 17x)+\frac{1}{271}\sin(2\pi\times 27x)-\frac{1}{703}\cos(2\pi\times 35x),\]
\[q(y) = \frac{1}{287}\cos(2\pi\times 15y)+\frac{1}{203}\sin(2\pi\times 27y)-\frac{1}{841}\sin(2\pi\times 38y).\]

On Figure~\ref{GrafDnDiffeoCons}, we have represented graphically the quantity $D(f_{128k})$ for $k$ from 1 to $150$. It appears that, as predicted by Theorem~\ref{limiteEgalZero}, this degree of recurrence goes to 0. In fact, it is even decreasing, and converges quite fast to 0: as soon as $N=128$, the degree of recurrence is smaller than $1/2$, and if $N\gtrsim 1000$, then $D(f_N) \le 1/10$. Note that, contrary to what is predicted by theory, this phenomenon was already observed for examples of conservative homeomorphisms of the torus which have big enough derivatives not to be considered as ``typical from the $C^1$ case'' (see \cite{Guih-discr}).
\bigskip

We also present the results of the numerical simulation we have conducted for the degree of recurrence of the expanding map of the circle $g$, defined by
\[g(x) = 2x + \varep_1 \cos(2\pi x) + \varep_2 \sin(6\pi x),\]
with $\varep_1 = 0.127\,943\,563\,72$ and $\varep_2 = 0.008\,247\,359\,61$.

On Figure~\ref{GrafDnExp}, we have represented the quantity $D(g_{128k})$ for $k$ from 1 to $1\,000$. It appears that, as predicted by Corollary~\ref{Tau0Dilat}, this degree of recurrence seems to tend to 0. In fact, it is even decreasing, and converges quite fast to 0: as soon as $N=128$, the degree of recurrence is smaller than $1/5$, and if $N\gtrsim 25\,000$, then $D(g_N) \le 1/50$.

\appendix 
\section*{Appendix}
\section{A more general setting where theorems remain true}\label{AddendSett}

Here, we give weaker assumptions under which the theorems of this paper are still true: the framework ``torus $\T^n$ with grids $E_N$ and Lebesgue measure'' could be seen as a little too restrictive.
\bigskip

So, we take a compact smooth manifold $M$ (possibly with boundary) and choose a partition $M_1,\cdots,M_k$ of $M$ into closed sets\footnote{That is, $\bigcup_i M_i = M$, and for $i\neq j$, the intersection between the interiors of $M_i$ and $M_j$ are empty.} with smooth boundaries, such that for every $i$, there exists a chart $\varphi_i : M_i\to \R^n$. We endow $\R^n$ with the euclidean distance, which defines a distance on $M$ \emph{via} the charts $\phi_i$ (this distance is not necessarily continuous). From now, we study what happens on a single chart, as what happens on the neighbourhoods of the boundaries of these charts ``counts for nothing'' from the Lebesgue measure viewpoint.

Finally, we suppose that the uniform measures on the grids $E_N = \bigcup_i E_{N,i}$ converge to a smooth measure $\lambda$ on $M$ when $N$ goes to infinity.

This can be easily seen that these conditions are sufficient for Corollary~\ref{CoroCoroArturJairo} to be still true.
\bigskip

For the rest of the statements of this paper, we need that the grids behave locally as the canonical grids on the torus.

For every $i$, we choose a sequence $(\kappa_{N,i})_N$ of positive real numbers such that $\kappa_{N,i}\underset{N\to +\infty}{\longrightarrow} 0$. This defines a sequence $E_{N,i}$ of grids on the set $M_i$ by $E_{N,i} = \varphi_i^{-1} (\kappa_{N,i}\Z^n)$. Also, the canonical projection $\pi : \R^n\to \Z^n$ (see Definition~\ref{DefDiscrLin}) allows to define the projection $\pi_{N,i}$, defined as the projection on $\kappa_{N,i}\Z^n$ in the coordinates given by $\varphi_i$:
\[\begin{array}{rcl}
\pi_{N,i} : & M_i & \longrightarrow E_{N,i}\\
            & x   & \longmapsto \varphi_i^{-1}\Big(\kappa_{N,i}\pi\big(\kappa_{N,i}^{-1}\varphi_i(x)\big)\Big).
\end{array}\]

We easily check that under these conditions, Theorems~\ref{convBisMieux}, \ref{TauxExpand} and \ref{limiteEgalZero} and Corollary~\ref{Tau0Dilat} are still true, that is if we replace the torus $\T^n$ by $M$, the uniform grids by the grids $E_N$, the canonical projections by the projections $\pi_{N,i}$, and Lebesgue measure by the measure $\lambda$.

\section*{Acknowledgments} 

I would like to warmly thank all of those who helped me during these researches, in particular the anonymous referees, Yves Meyer, who made me discover model sets, and Fran{\c c}ois B{\'e}guin for constant support and numerous proofreadings. Many thanks to Fran{\c c}ois Petit for reference \cite{MR0217337}.


\bibliographystyle{amsplain}
\providecommand{\MR}[1]{}
\providecommand{\bysame}{\leavevmode\hbox to3em{\hrulefill}\thinspace}
\providecommand{\MR}{\relax\ifhmode\unskip\space\fi MR }
\providecommand{\MRhref}[2]{%
  \href{http://www.ams.org/mathscinet-getitem?mr=#1}{#2}
}
\providecommand{\href}[2]{#2}


\begin{dajauthors}
\begin{authorinfo}[pag]
  Pierre-Antoine Guih{\'e}neuf\\
  Universit\'e Paris-Sud 11 / Universidade Federal Fluminense (Niter{\`o}i)\\
  Current address: IMJ-PRG, 4 place Jussieu, case 247, 75252 Paris Cedex 05, France\\
  pierre-antoine.guiheneuf \imageat{}imj-prg\imagedot{} fr\\
  \url{https://webusers.imj-prg.fr/~pierre-antoine.guiheneuf/}
\end{authorinfo}

\end{dajauthors}

\end{document}